\documentclass[11pt]{amsart}
\usepackage{amssymb}
\usepackage{latexsym}
\usepackage{graphicx}
\usepackage{tikz}
\usepackage[linktocpage=true]{hyperref}
\usepackage[left=3.8cm, right=3.8cm, top=3cm, bottom=3cm]{geometry}
\usepackage{color}
\usepackage{here}
\usepackage[backgroundcolor = lightgray,
linecolor = lightgray,
textsize = tiny]{todonotes}
\usepackage{mathrsfs}
\usepackage[all]{xy}
\usepackage{enumitem}

\hypersetup{ colorlinks   = true, %Colours links instead of ugly boxes
	urlcolor  = blue, %Colour for external hyperlinks
	linkcolor    = blue, %Colour of internal links
	citecolor   = blue %Colour of citations
}

\def\beq{\begin{equation}}
\def\eeq{\end{equation}}

\def\det{\mathrm{det}\ }

\def\ba{\bar a}
\def\bg{\bar{\gamma}}

\def\bgg{\bar{g}}

\newcommand{\billiards}{\mathbf{B}}

\newcommand{\Z}{{\mathbb Z}}
\newcommand{\R}{{\mathbb R}}

\newcommand{\T}{{\mathbb T}}

\newcommand{\N}{{\mathbb N}}

\newtheorem{theorem}{Theorem}[section]
\newtheorem{remark}[theorem]{Remark}

\newtheorem{lemma}[theorem]{Lemma}
\newtheorem{conj}[theorem]{Conjecture}
\newtheorem{defi}[theorem]{Definition}
\newtheorem{prop}[theorem]{Proposition}
\newtheorem{corollary}[theorem]{Corollary}

\newtheorem*{main thm}{Main Theorem}
\newtheorem*{main thm bis}{Main Theorem (alternate version)}

\newtheorem{theoalph}{Theorem}

\newtheorem{question}{Question}
\def\ve{\varepsilon}
\def\htop{h_{\operatorname{top}}}
\newcommand{\dom}{\mathcal D}
\newcommand{\obs}{\mathcal O}

\begin{document}
	\numberwithin{equation}{section}
	
	\title[Entropy rigidity for 3D Anosov flows and dispersing billiards]{Entropy rigidity for 3D conservative Anosov flows and dispersing billiards}
	
	\author{Jacopo De Simoi}
	\address{Jacopo De Simoi, Department of Mathematics, University of Toronto, Toronto,
		ON, Canada M5S2E4}
	\email{jacopods@math.utoronto.ca}

	\author[Martin Leguil]{Martin Leguil$^*$}
	\thanks{$^*$M.L. was supported by the ERC project 692925 NUHGD of Sylvain Crovisier.}
	\address{Martin Leguil, CNRS-Laboratoire de Math\'ematiques d’Orsay, UMR 8628, Universit\'e Paris-Sud 11, Orsay Cedex 91405, France}
	\email{martin.leguil@math.u-psud.fr}

	\author{Kurt Vinhage}
	\address{Kurt Vinhage, Department of Mathematics, The Pennsylvania State University, State College, PA, United States 16802}
	\email{kwv104@psu.edu}
	
	\author{Yun Yang}
	\address{Yun Yang, Department of Mathematics, Virginia Polytechnic Institute and State University, VA, United States 24060}
	\email{yunyang@vt.edu}
	
	\begin{abstract}
	Given an integer $k \geq 5$, and a $C^k$    Anosov flow $\Phi$ on some compact connected $3$-manifold preserving a smooth volume, we show that  the measure of maximal entropy (MME) is   the volume measure 
	 if and only if $\Phi$ is  $C^{k-\varepsilon}$-conjugate to an algebraic flow, for $\varepsilon>0$   arbitrarily small. Besides the rigidity, we also study the entropy flexibility, and show that  the metric entropy with respect to the volume measure and the topological entropy of suspension flows over Anosov diffeomorphisms on the $2$-torus achieve all possible values subject to natural normalizations. Moreover, in the case of dispersing billiards, we  show that if the measure of maximal entropy is   the volume measure, then the Birkhoff Normal Form of regular periodic orbits with a homoclinic intersection  is linear.  
	\end{abstract}
	
	\maketitle
	\tableofcontents

\section{Introduction}
	
\emph{Anosov flows} and \emph{Anosov diffeomorphisms} are among the most well-understood dynamical systems, including the space of invariant measures, stable and unstable distributions and foliations, decay of correlations and other statistical properties. The topological classification of Anosov systems, especially diffeomorphisms, is well-understood in low dimensions. For flows, constructions of ``exotic'' Anosov flows are often built from gluing several examples of algebraic or geometric origin. Special among them are the {\it algebraic systems}, affine systems on homogeneous spaces. In the diffeomorphism case, these are automorphisms of tori and nilmanifolds. Conjecturally, up to topological conjugacy, these account for all Anosov diffeomorphisms (up to finite cover).
The case of Anosov flows is quite different. Here, the algebraic models are suspensions of such diffeomorphisms and geodesic flows on negatively curved rank one symmetric spaces. There are many constructions to give new Anosov flows, all of which come from putting together geodesic flows and/or suspensions of diffeomorphisms. In particular, quite unexpected behaviors are possible, including Anosov flows on connected manifolds which are not transitive \cite{FW}, analytic Anosov flows that are not algebraic \cite{HT}, contact Anosov flows on hyperbolic manifolds that are not
topologically orbit equivalent to an algebraic flow \cite{FH}, and many other constructions combining these ideas (see, e.g., \cite{BBY}).
	
When classifying Anosov systems up to $C^\infty$ diffeomorphisms, the question becomes very different. Here, the algebraic models are believed to distinguish themselves in many ways, including regularity of dynamical distributions and thermodynamical formalism. In the case of the latter, the first such result was obtained by Katok for geodesic flows of negatively curved surfaces, where it was shown that coincidence of metric entropy with respect to Liouville measure and topological entropy implies that the surface has constant negative curvature \cite{K1,K3}.  %{\color{red} corrected phrasing}.
These works lead to the following conjecture:

\begin{conj}[Katok Entropy Conjecture]
		\label{conj:katok-conj}
		Let $(M,g)$ be a connected Riemannian manifold of negative curvature, and $\Phi$ be the corresponding geodesic flow. Then $\htop(\Phi) = h_{\mu}(\Phi)$ if and only if $(M,g)$ is a locally symmetric space.
\end{conj}

A weaker version of this was obtained in higher dimensional geodesic flows under negative curvature assumptions by \cite{BCG}, which still highly depends on the structures coming from the geometry of the flow. Other generalizations work with broader classes of Anosov flows. Foulon \cite{Fo} showed that in the case of a contact Anosov flow $\Phi$ on a closed three manifold, $\Phi$ is, up to finite cover,  smoothly conjugate to geodesic flow of a metric of constant negative curvature on a closed surface if and only if the measure of maximal entropy is the contact volume. There, he asks the following question generalizing Conjecture \ref{conj:katok-conj}:%In the case of general Anosov flows on compact %Riemannian 
%$3$-manifolds, one must first assume the existence of  a smooth invariant volume measure, which must be the \emph{SRB measure} (see Subsection \ref{subsection equi} for more about SRB measures). It is therefore natural to ask the following question: 

\begin{question}
		\label{que:entropy-rigidity}
		Let $\Phi$ be a smooth Anosov flow on a 3-manifold   which preserves a smooth volume $\mu$. If $\htop(\Phi) = h_{\mu}(\Phi)$, is $\Phi$ smoothly conjugate to an algebraic flow?
\end{question}

Let us briefly compare Question \ref{que:entropy-rigidity} and Conjecture \ref{conj:katok-conj}. Recall that the geodesic flow on $M$ occurs on the unit tangent bundle, $T^1M$, which has dimension $2\dim M- 1$. Therefore, Question \ref{que:entropy-rigidity} corresponds to the case of geodesic flows on surfaces, which was proved by Katok in \cite{K1,K3}.

The low-dimensionality assumption of Question \ref{que:entropy-rigidity} is required for a theorem in this generality. It is not difficult to construct non-algebraic systems whose maximal entropy measure is a volume when the stable and unstable distributions are multidimensional. This highlights the power of assuming the dynamics is a geodesic flow and why a solution to Conjecture \ref{conj:katok-conj} would require a mixture of dynamical and geometric ideas.

%We avoid one of the central difficulties when approaching Conjecture \ref{conj:katok-conj} more generally by limiting ourselves to dimension 3. However, the lack of geometric structure with which to tame the stable and unstable distributions requires us to develop new methods to answer the question fully. %Foulon answered Question \ref{que:entropy-rigidity} in the case of contact flows in \cite{Fo}. The invariant contact form was used to provide key regularity and non-integrability properties of the (strong) stable and unstable distributions.

By contrast, Question \ref{que:entropy-rigidity} is purely dynamical. In particular, it applies to time changes or perturbations of the Handel-Thurston flow \cite{HT} as well as special flows over Anosov diffeomorphisms of $\T^2$, for which no entropy-rigidity type results were known since they do not preserve contact forms (except for very special cases, see \cite{FH}).

We provide a positive answer to Question \ref{que:entropy-rigidity}.

\begin{theoalph}\label{thm:main bis}
Let $k \geq 5$ be some integer, and let $\Phi$ be a $C^k$  Anosov flow on a compact connected %Riemannian
$3$-manifold  $M$ such that $\Phi_* \mu=\mu$  for some smooth volume $\mu$. Then $\htop(\Phi) = h_\mu(\Phi)$  if and only if $\Phi$ is $C^{k-\varepsilon}$-conjugate to an algebraic flow, for   $\varepsilon>0$  arbitrarily small. %, where $\mu$ denotes the volume measure on $M$. 
\end{theoalph}
	
%In the case of suspensions, we obtain:

	%\begin{theorem}
	%	\label{thm:main}
%		Let $\Phi$ be a $C^\infty$ suspension of a $C^{\infty}$ area-preserving Anosov diffeomorphism $f$ of $\T^2$. Then $\Phi$ is $C^\infty$-conjugate to a constant-time suspension of a hyperbolic automorphism of $\T^2$ if and only if and only if $\htop(\Phi) = h_\mu(\Phi)$, where $\mu$ is the smooth invariant volume of the suspension manifold over $\T^2$.
%	\end{theorem}

%	\begin{remark}
%		Many of the assumptions of Theorem \ref{thm:main} can be replaced by others. In particular, one may assume that $f$ is an Anosov diffeomorphism of a surface preserving a volume, one may conclude that the surface is $\T^2$ \cite{N} and a coordinate system for which the invariant measure is the standard area form by the Darboux theorem. Furthermore, an Anosov flow is a suspension whenever it is non-mixing \cite{PaPo}, jointly integrable \cite{Pl} or non-accessible \cite{BPW}.
%	\end{remark}

%{\color{red} removed the suspension theorem, I don't think it's needed}

	%\subsection{Comments of the Proof}

Moreover, we believe the theorem is true for $k \ge 2$, but technical obstructions prevent us from finding the precise boundary of required regularity. 
Let us emphasize that regularity is extremely important for the rigidity phenomenon. If one relaxes to the $C^1$ category, it is possible to produce examples of flows with quite different behaviour. Indeed,  given a $C^2$ Axiom A flow   on a compact Riemannian manifold  and an \emph{attractor}  (see Subsection \ref{general facts} for a definition) whose unstable distribution is $C^1$, Parry \cite{Pa} describes a $C^1$ time change such that for the new flow, the SRB measure (invariant volume if it exists) of  the attractor coincides with the measure of maximal entropy. These measures are obtained as limits of certain closed orbital measures. In particular, for any $C^2$ transitive Anosov flow $\Phi$ on a $3$-manifold,  the unstable distribution is $C^1$ (see Remark \ref{remarque reg}), hence the synchronization procedure explained by Parry shows that $\Phi$ is $C^1$-orbit equivalent to a $C^1$ Anosov flow for which the SRB measure is equal to the measure of maximal entropy. 
	
Conversely, recent results of Adeboye, Bray and Constantine \cite{ABC} show that systems with more geometric structure still exhibit rigidity in low regularity. In particular, they show a version of the Katok entropy conjecture for geodesic flows on Hilbert geometries, which are only $C^{1+\alpha}$ flows. 
	
Question \ref{que:entropy-rigidity} can be modified to remove the volume preservation assumption. Given any transitive Anosov flow, there are always two natural measures to consider (which fit into a broader class of measures called {\it Gibbs states}). One is the measure of maximal entropy, which is the unique measure $\nu$ such that $h_\nu(\Phi) = \htop(\Phi)$. The other is the Sinai-Ruelle-Bowen, or SRB measure, which has its conditional measures along unstable manifolds absolutely continuous with respect to Lebesgue. It is therefore natural to consider the following:

\begin{question}
\label{que:SRB-rigidity}
Suppose that $\Phi$ is a $C^\infty$ transitive Anosov on a 3-manifold. If the SRB measure and MME coincide, is $\Phi$ smoothly conjugate to an algebraic flow?
\end{question}

Our approach is insufficient to answer Question \ref{que:SRB-rigidity}, because our method relies on some special change of coordinates introduced by  Hurder-Katok \cite{HuKa} for volume-preserving Anosov flows, and on \emph{Birkhoff Normal Forms} for area-preserving maps (see \eqref{birk nor form} below).

\subsection{Outline of arguments and new techniques}

The techniques used to prove Theorem \ref{thm:main bis} combine several ideas. We do not directly follow the approach of Foulon, who aims to construct a homogeneous structure on the manifold by using a Lie algebra of vector fields tangent to the stable and unstable distributions, together with the vector field generating the flow \cite{Fo}.  %We instead follow an unpbulished approach of Jiagang Yang. 
The central problem in all approaches is the regularity of dynamical distributions. Foulon requires smoothness of the {\it strong} stable and unstable distributions, which follows a posteriori from the existence of a smooth conjugacy to an algebraic model, but is often out of reach without additional assumptions, such as a smooth invariant $1$-form. %Smoothness of the strong distributions are in fact equivalent to the existence of a smooth conjugacy to an algebraic model.

The main technical result of the paper is the smoothness of the \emph{weak} foliations. For this we use an invariant of Hurder and Katok \cite{HuKa}, the \emph{Anosov class}, which is an idea that can be dated back to Birkhoff and Anosov. In fact, Anosov used this idea to show the existence of examples for which the regularity of the distributions is low: it was first shown that the invariant is non-trivial, but that it must be trivial in all algebraic examples. Hurder and Katok  \cite{HuKa} then proved that vanishing of the invariant implies smoothness of the weak stable and unstable distributions. In the case of an Anosov flow obtained by suspending an Anosov diffeomorphism of the $2$-torus, they showed that it implies   smooth conjugacy to an algebraic model.  For more regarding the  Anosov  class and its connection to \emph{normal forms} for hyperbolic maps, see Subsection \ref{subs anosov class}.

For a transitive Anosov flow on some compact connected $3$-manifold, there exists a unique \emph{measure of maximal of entropy}, or \emph{MME}. If the MME is equal to the volume measure, we are able to show vanishing of the Anosov class. We deduce from \cite{HuKa} that  the weak foliations are smooth.  This gives a weaker form of rigidity, due to Ghys \cite{G2}: the existence of a smooth orbit equivalence to an algebraic model. This is not surprising, as taking a time change of any Anosov flow will preserve its weak foliations. With this in hand, there is one last lemma to prove: any time change of an algebraic Anosov flow for which the measure of maximal entropy is equivalent to Lebesgue is smoothly conjugate to a linear time change. This is proved in Proposition \ref{main prop yang}.

The difficulty in proving smoothness using local normal forms is that transverse regularity of stable foliations can't be controlled locally. We circumvent this problem by using an orbit $h_\infty$  homoclinic to a reference periodic orbit $\mathcal{O}$, and a sequence $(h_n)_{n \geq 0}$ of periodic orbits  with prescribed combinatorics shadowing the   orbit $h_\infty$. We are able to control the dynamical and differential properties of the orbits $(h_n)_{n \geq 0}$  by choosing an exceptionally good chart (see Subsection \ref{subsec:horseshoe-homoclinic}). 
On the one hand, in Subsection \ref{subs periodic exp}, we show that for an \emph{Axiom A flow} restricted to some \emph{basic set} (see Subsection \ref{general facts} for the definitions), the equality of the MME  and of the SRB measure forces the periodic Lyapunov exponents to be equal; by controlling the periods of the orbits $(h_n)_{n \geq 0}$, this allows us to obtain a first estimate on the Floquet multipliers of the flow for $h_n$, $n \geq 0$. On the other hand, by the choice of the orbits  $(h_n)_{n \geq 0}$, we can also study the asymptotics of the Floquet multipliers through the Birkhoff Normal Form of the periodic orbit $\mathcal{O}$ (see Subsection \ref{subsec:horseshoe-homoclinic}); the calculations are similar to those in \cite{DKL}. Combining those two estimates, we show the following:
\begin{theoalph}\label{theorem b}
	Let $k \geq 5$ be an integer, and let $\Phi$ be a $C^k$  Anosov flow on a $3$-manifold $M$ which preserves a smooth volume $\mu$. If $\htop(\Phi) = h_\mu(\Phi)$, then the Anosov class vanishes identically. 
	Moreover, this implies strong rigidity properties: 
	\begin{enumerate}
		\item\label{point un rig} the weak stable/unstable distributions of the flow $\Phi$ are of class $C^{k-3}$;
		\item\label{point deux rig}  the flow $\Phi$ is $C^k$-orbit equivalent to an algebraic model. 
	\end{enumerate}
\end{theoalph} 

Item \eqref{point un rig} follows from the vanishing of the Anosov class and the work of Hurder-Katok \cite{HuKa}, while   \eqref{point deux rig} follows from \eqref{point un rig} and the work of Ghys \cite{G2}. 

In order to upgrade the orbit equivalence to a smooth flow conjugacy, we use ideas similar to those in Subsection 3.3 of an unpublished paper \cite{Ya} of Yang; indeed, the equality of the MME and of the volume measure allows us to discard bad situations such as those described in  \cite{G1}. In Subsection \ref{subs orb conj}, we will prove:

\begin{theoalph}\label{theorem c}
	Let $\Phi_0$ be an   Anosov flow on some $3$-manifold that is a smooth time change of an algebraic flow $\Psi_0$. If the measure of maximal entropy of $\Phi_0$ is absolutely continuous with respect to volume, then $\Phi_0$ is a linear time change of $\Psi_0$.
\end{theoalph}

In other words, up to a linear time change, the periods of associated periodic orbits for the flow and the algebraic model  coincide. By Livsic's theorem, this allows us to synchronize the orbit equivalence, and produce a conjugacy between the two flows; the smoothness of this conjugacy is automatic (it follows from previous rigidity results in \cite{dlL}).  In particular, Theorem \ref{thm:main bis} is a consequence of Theorems \ref{theorem b} and \ref{theorem c}.

\subsection{Entropy flexibility}

Besides the rigidity, in Section \ref{sec:flexibility}, we also study the entropy flexibility for suspension Anosov flows following the program in \cite{EK} for Anosov systems. We show that  the metric entropy with respect to the volume measure and the topological entropy of suspension flows over Anosov diffeomorphisms on the $2$-torus achieve all possible values subject to two natural normalizations. 
\begin{theoalph}\label{prop:flexibility}
	Let $A \in \mathrm{SL}(2,\Z)$ be a hyperbolic matrix whose induced torus automorphism has topological entropy $h>0$. Let $\mu$ be the volume measure on $\T^2$. Then for any $c_{\operatorname{top}}  > c_\mu> 0$ such that $c_\mu \le h$, there exists a volume-preserving Anosov diffeomorphism $f\colon \T^2 \to \T^2$  homotopic to $A$ and a $C^\infty$ function $r \colon \T^2 \to \R^{+}$  such that $\int r \, d\mu = 1$ and if $\Phi$ is the suspension flow induced by $f$ and $r$, then $\htop(\Phi) = c_{\operatorname{top}}$ and $h_\mu(\Phi) = c_\mu$.\footnote{By a slight abuse of notation, we identify $\mu$ with the induced measure on the suspension space.}
\end{theoalph}
\begin{theoalph}\label{prop:flexibilitymme}
	Let $A \in \mathrm{SL}(2,\Z)$ be a hyperbolic matrix whose induced torus automorphism has topological entropy $h>0$. Let $\mu$ be the volume measure on $\T^2$. Then for any $c_{\operatorname{top}} > c_\mu> 0$ such that $c_{\operatorname{top}}\geq h$, there exists a volume-preserving Anosov diffeomorphism $f\colon \T^2 \to \T^2$  homotopic to $A$ with maximal entropy measure $\nu$ and a $C^\infty$ function $r \colon \T^2 \to \R^{+}$  such that $\int r \, d\nu = 1$ and if $\Phi$ is the suspension flow induced by $f$ and $r$, then $\htop(\Phi) = c_{\operatorname{top}}$ and $h_\mu(\Phi) = c_\mu$.\end{theoalph}

Theorem \ref{prop:flexibility} and  
Theorem \ref{prop:flexibilitymme} are optimal -- see Section \ref{sec:flexibility} for the discussion on why they are optimal.

\subsection{Entropy rigidity for dispersing billiards}

In the last section of this paper (Section \ref{section billiards}), we investigate the case of \emph{dispersing billiards}. The dynamics of such billiards is  hyperbolic; moreover, they preserve a smooth volume measure, hence they can be seen as an analogue to the  conservative Anosov flows on $3$-manifolds considered previously; yet, due to the possible existence of grazing collisions, the billiard map has singularities. An orbit with  no tangential collisions is called \emph{regular}. 
As the billiards under consideration are hyperbolic,   recall that for any point $x$ in a regular orbit $\mathcal{O}$ of period $p \geq 1$ of the billiard map, there exists    a neighbourhood of $x$ where the $p$-th iterate of the billiard map can be   conjugate through a $C^\infty$ volume-preserving local diffeomorphism  to a unique   map
\begin{equation}\label{birk nor form}
N\colon (\xi,\eta)\mapsto(\Delta(\xi\eta)\cdot\xi,\Delta(\xi\eta)^{-1}\cdot\eta),
\end{equation}
called the  \emph{Birkhoff Normal Form}, where 
$\Delta\colon z \mapsto \sum_{k=0}^{+\infty} a_k z^k$, $|a_0|\in (0,1)$. For each $k \geq 0$, we call $a_k=a_k(\mathcal{O})$ the $k$-th \emph{Birkhoff invariant} of $\mathcal{O}$, and we say that $N$ is  \emph{linear} if $a_k =0$, for all $k \geq 1$. 

In Subsection \ref{seubc dis open}, we study a class of \emph{open dispersing billiards} satisfying a \emph{non-eclipse condition}; this class has already been considered in many works  (\cite{GR,Mor,Stoy1,Stoy2,Stoy3,PS,BDKL,DKL}\dots). The  dynamics  of the billiard flow is of type Axiom A, and there exists a  unique \emph{basic set} (see Subsection \ref{general facts} for the definitions). In particular, it has a unique MME and a unique SRB measure, and we show that equality of the MME with the SRB measure forces the first Birkhoff invariant of any periodic orbit to vanish. 

In Subsection \ref{sinnnnnai}, we study  $C^\infty$ Sinai billiards with \emph{finite horizon}. % i.e., such that no trajectory  makes only tangential collisions. 
In this case, we consider the discrete dynamics, i.e., the billiard \emph{map}, which allows us to use  the work of Baladi and Demers \cite{BD}. Although the billiard map  has singularities, Baladi-Demers  \cite{BD} were able to define a suitable notion of topological entropy $h_*$. Moreover, assuming  some quantitative lower bound on $h_*$, they showed that there exists a unique invariant Borel probability measure $\mu_*$ of maximal entropy. 
Our result says that the equality of the MME and of the natural volume measure $\mu$ imposes strong restrictions on the local dynamics: %,%can then be stated as follows; roughly speaking, it says that if the MME is equal to the natural volume measure $\mu$, then the dynamics has strong ``homogeneity'' properties. 

\begin{theoalph}\label{main theo billiards}
	Fix a $C^\infty$ Sinai billiard  with finite horizon satisfying the quantitative condition in \cite{BD} (see \eqref{cond h etoile}). 
	If the measure of maximal entropy $\mu_*$ of the billiard map is equal to the   volume measure $\mu$, then for any regular periodic orbit with a homoclinic intersection, the associated Birkhoff Normal Form is linear. 
\end{theoalph}
 
In \cite{BD}, the authors show that for $\mu_*$ to be equal to $\mu$, it is necessary that  the Lyapunov of regular periodic orbits are all equal to $h_*$ (see Proposition 7.13 in \cite{BD}), and observe that no dispersing billiards with this property are known. Compared with theirs, the necessary condition we derive for Birkhoff Normal Forms is more local. Although there is some flexibility for the Birkhoff Normal Forms which can be realized -- at least formally -- (see for instance the work of Treschev or of Colin de Verdi\`ere \cite[Section 5]{CdV} in the convex case), the fact it is linear imposes strong geometric restrictions (by \cite{CdV}, for symmetric $2$-periodic orbits,  the Birkhoff Normal Form encodes the jet of the curvature at the bouncing points), and we expect that no Sinai billiard satisfies the conclusion of Theorem \ref{main theo billiards}. 

\vspace{.5cm}

{\noindent \it Acknowledgements.} This paper has gone through several stages, and revisions of the author list. The last two authors would like to thank Cameron Bishop and David Hughes for their work on earlier iterations of these results, in particular related to our understanding of the Anosov class. They would also like to thank the AMS Mathematical Resarch Communities program {\it Dynamical Systems: Smooth, Symbolic and Measurable}, where several of the initial ideas and work was made on this project. The program also provided travel funds for further collaboration.

The authors would like to thank Viviane Baladi, Aaron Brown, Sylvain Crovisier, Alena Erchenko, Andrey Gogolev, Boris Hasselblatt,  Carlos Matheus, Rafael Potrie, Federico Rodriguez-Hertz, Ralf Spatzier and Amie Wilkinson for discussions and encouragement on this project. Finally, we would like to acknowledge an unpublished preprint of Jiagang Yang where we learned several useful  ideas for the argument in Subsection \ref{subs orb conj}.

\section{Preliminaries}
	
	\subsection{General facts about hyperbolic flows}\label{general facts}
	
	Let us first  recall some classical facts about Anosov and Axiom A flows.
	
	In the following, we fix a  $C^\infty$ smooth compact Riemannian manifold $M$, and we consider a  $C^2$  flow  $\Phi=(\Phi^t)_{t \in \R}$ on $M$.  We denote by $X_\Phi=X\colon x \mapsto \frac{d}{dt}|_{t=0} \Phi^t(x)$   the vector field tangent to the direction of the flow.  %; we follow the presentation given in \cite{C}. 
	Recall that the \emph{nonwandering set} $\Omega_\Phi=\Omega \subset M$ is the set of points $x$ such that for any open set $U \ni x$, any $T_0>0$, there exists $T>T_0$ such that $\Phi^T (U) \cap U \neq \emptyset$. 
	
	\begin{defi}[Hyperbolic set]
		A $\Phi$-invariant subset $\Lambda \subset M$ without fixed points is called a \textit{(uniformly) hyperbolic} set if there exists a $D\Phi$-invariant splitting 
		$$
		T_x M=E^s(x) \oplus \R X(x) \oplus E^u(x),\qquad \forall\, x \in \Lambda,
		$$
		where the \emph{(strong) stable bundle} $E_\Phi^s=E^s$, resp. the \emph{(strong) unstable bundle} $E_\Phi^u=E^u$ is uniformly contracted, resp. expanded, i.e., for some constants $C>0$, $0 \leq \theta<1$, it holds 
		\begin{align*}
		\| D_x\Phi^t \cdot v\| &\leq C \theta^t \|v\|,\qquad \forall\, x \in \Lambda,\, \forall\, v\in E^s(x),\, \forall\, t \geq 0, \\
		\| D_x\Phi^{-t} \cdot v\| &\leq C \theta^t \|v\|,\qquad \forall\, x \in \Lambda,\,  \forall\, v \in E^u(x),\, \forall\, t \geq 0. 
		\end{align*}
		We also denote by $E_\Phi^{cs}=E^{cs}$, resp. $E_\Phi^{cu}=E^{cu}$, the \emph{weak stable bundle} $E^{cs}:=E^s \oplus \R X$, resp. the \emph{weak unstable bundle}  $E^{cu}:=\R X\oplus E^u$. 
	\end{defi}
	
	\begin{defi}[Anosov/Axiom A flow]\quad\\
		\indent $\bullet$	A flow $\Phi\colon M \to M$ is called an \emph{Anosov flow} if the entire manifold $M$ is a hyperbolic set. In this case, the stable bundle $E_\Phi^s=E^s$, resp. the unstable bundle $E_\Phi^u=E^u$ integrates to a continuous foliation $\mathcal{W}^s_\Phi=\mathcal{W}^s$, resp. $\mathcal{W}_\Phi^{u}=\mathcal{W}^u$, called the \emph{(strong) stable foliation}, resp. the \emph{(strong) unstable foliation}. Similarly, $E^{cs}$, resp. $E^{cu}$ integrates to a continuous foliation $\mathcal{W}_\Phi^{cs}=\mathcal{W}^{cs}$, resp. $\mathcal{W}_\Phi^{cu}=\mathcal{W}^{cu}$, called the \emph{weak stable foliation}, resp. the \emph{weak unstable foliation}. Moreover, each of these foliations is invariant under the dynamics, i.e., $\Phi^t(\mathcal{W}^*(x))=\mathcal{W}^*(\Phi^t(x))$, for all $x \in M$ and $*=s,u,cs,cu$.

		\indent $\bullet$ A flow $\Phi \colon M \to M$ is called an \emph{Axiom A flow} if the nonwandering set $\Omega \subset M$ can be written as a disjoint union $\Omega=\Lambda \cup F$, where $\Lambda$ is a closed hyperbolic set where periodic orbits are dense,  and $F$ is a finite union of hyperbolic fixed points. 
	\end{defi}

\begin{remark}\label{remarque reg}
	Let  $\Phi \colon M \to M$ be a $C^2$  Anosov flow on some compact connected manifold $M$. 
	
	\indent $\bullet$ In general the distributions $E_{\Phi}^{*}$, $*=s,u,cs,cu$, are only H\"{o}lder continuous, but  when $\dim M=3$, Hirsch-Pugh \cite{HP} have shown that   $E_{\Phi}^{cs}$ and $E_{\Phi}^{cu}$ are of class $C^1$.  
	
	\indent $\bullet$ As we will  sometimes need to assume topological transitivity, let us recall that in the case under consideration, namely, when $\Phi$ is conservative, transitivity is automatic. 
\end{remark}

% \footnote{For instance, any such flow is \emph{accessible} (any two points can be connected by a path ), and accessibility together with non-wandering assumptions imply topological transivitiy.}

	\begin{theorem}
		For an Axiom A flow $\Phi$ with a decomposition $\Omega=\Lambda \cup F$ as above, we have $\Lambda=\Lambda_1\cup \cdots \cup \Lambda_m$ for some integer $m \geq 1$, where for each $i \in \{1,\dots,m\}$,  $\Lambda_i$ is a closed $\Phi$-invariant hyperbolic set such that  $\Phi|_{\Lambda_i}$ is transitive, and $\Lambda_i=\cap_{t \in \R} \Phi^t (U_i)$ for some open set $U_i \supset \Lambda_i$. The set $\Lambda_i$ is called a \emph{basic set} of $\Phi$. A basic set $\Lambda_i$ is called an \emph{attractor} if $\Lambda_i=\cap_{t>0} \Phi^t (V_i)$ for some open set $V_i \supset \Lambda_i$. 
	\end{theorem}
	
	\begin{defi}[Algebraic flows -- see Tomter \cite{T}] An Anosov flow $\Phi \colon M \to M$ on a $3$-dimensional compact manifold $M$ is \emph{algebraic} if it is finitely covered by 
		\begin{enumerate}
			\item a suspension of a hyperbolic automorphism of the $2$-torus $\T^2:=\R^2 / \Z^2$;
			\item or the geodesic flow on some closed Riemannian surface of constant negative  curvature, i.e., a flow on a homogeneous space $\Gamma\backslash \widetilde{\mathrm{SL}(2,\R)}$ corresponding to right translations  by diagonal matrices $\mathrm{diag}(e^{t} ,e^{-t})$, $t \in \R$, where $\widetilde{\mathrm{SL}(2,\R)}$ denotes the universal cover of $\mathrm{SL}(2,\R)$, and $\Gamma$ is a uniform subgroup.
		\end{enumerate}

	%	or some finite factor of these actions.
	\end{defi}

\subsection{Equilibrium states for Anosov/Axiom A flows}\label{subsection equi}

In the following, we recall some classical facts about the equilibrium states of transitive Anosov flows/Axiom A flows, following the presentation given in \cite{C}. For simplicity, we consider the case of a $C^2$ transitive Anosov flow  $\Phi\colon M \to M$ on some $C^\infty$ smooth compact Riemannian manifold, but up to some minor technical details, everything goes through as well for the restriction $\Phi|_{\Lambda}$ of some $C^2$ Axiom A flow $\Phi$ to some basic set $\Lambda$.    %($\Lambda=M$ when  $\Phi$ is Anosov). 
	
	\begin{defi}[Rectangle, proper family]\quad \\
		\indent $\bullet$	A closed subset $R \subset M$ is called a \emph{rectangle} if there is a small closed codimension one smooth disk $D \subset M$ transverse to the flow $\Phi$ such that $R \subset D$, and for any $x,y \in R$, the point 
		$$
		[x,y]_R:=D \cap \mathcal{W}_{\Phi,\mathrm{loc}}^{cs}(x)\cap \mathcal{W}_{\Phi,\mathrm{loc}}^{cu}(y)
		$$
		exists and also belongs to $R$. A  rectangle $R$ is called \emph{proper} if $R=\overline{\mathrm{int}(R)}$ in the topology of $D$. For any rectangle $R$ and any $x \in R$, we let
		$$
		\mathcal{W}_R^s(x):= R \cap \mathcal{W}_{\Phi,\mathrm{loc}}^{cs}(x),\quad \mathcal{W}_R^u(x):= R \cap \mathcal{W}_{\Phi,\mathrm{loc}}^{cu}(x).
		$$
		
		$\bullet$ A finite collection of proper rectangles $\mathcal{R}=\{R_1,\dots,R_m\}$, $m \geq 1$, is called a \emph{proper family of size} $\epsilon>0$ if 
		\begin{enumerate}
			\item $M=\{\Phi^{t}(\mathcal{S}):t \in [-\epsilon,0]\}$, where $\mathcal{S}:=R_1\cup \dots \cup R_m$;
			\item $\mathrm{diam}(D_i)< \epsilon$, for each $i=1,\dots,m$, where $D_i\supset R_i$ is a disk   as above;
			\item for any $i\neq j$, $D_i\cap \{\Phi^{t}(D_j):t \in [0,\epsilon]\}=\emptyset$ or $D_j\cap \{\Phi^{t}(D_i):t \in [0,\epsilon]\}=\emptyset$. 
		\end{enumerate}
		The set $\mathcal{S}$ is called a \emph{cross-section} of the flow $\Phi$; it is associated with a Poincar\'e map $\mathcal{F}\colon \mathcal{S}\to \mathcal{S}$, where for any $x \in \mathcal{S}$, $\mathcal{F}(x)=\Phi^{\tau(x)} (x)$,  the function $\tau\colon \mathcal{S}\to \R_+$ being the first return time on $\mathcal{S}$. 
	\end{defi}

 In the following, given a proper family  $\mathcal{R}=\{R_1,\dots,R_m\}$, $m \geq 1$, for $*=s,u$, and for any $x \in R_i$, $i\in \{1,\dots,m\}$, we will also set $\mathcal{W}_{\mathcal{F}}^*(x)=\mathcal{W}_{\mathcal{F},\mathrm{loc}}^*(x):=\mathcal{W}_{R_i}^*(x)$.

	\begin{defi}[Markov family]
		Given $\epsilon>0$ small and $m\geq 1$, a proper family $\mathcal{R}=\{R_1,\dots,R_m\}$ of size $\epsilon$, with Poincar\'e map $\mathcal{F}$, is called a \emph{Markov family} if it satisfies the following Markov property:
 for any $x \in \mathrm{int} (R_i)  \cap \mathcal{F}^{-1}(\mathrm{int} (R_j))$, $i,j\in \{1,\dots,m\}$, it holds 
		$$
		\mathcal{W}_{R_i}^s(x)\subset \overline{\mathcal{F}^{-1}(\mathcal{W}_{R_j}^s(\mathcal{F}(x)))}\quad \text{and}\quad \overline{\mathcal{F}(\mathcal{W}_{R_i}^u(x))}\supset \mathcal{W}_{R_j}^u(\mathcal{F}(x)).
		$$
	\end{defi}

	\begin{theorem}[see Theorem 4.2 in \cite{C}]
		Any transitive Anosov flow has a Markov family of arbitrary small size; the same is true for the restriction of an Axiom A flow to any basic set. 
	\end{theorem}

	\begin{prop}[see Proposition 4.6 in \cite{C}]
		Let $p \colon M \to \R$ be a H\"{o}lder continuous function. Then there exists a unique \emph{equilibrium state} $\mu_p$ for the potential $p$; in other words, $\mu_p$ is the unique $\Phi$-invariant measure which achieves the supremum
		$$
		P_p:=\sup_{\Phi_*\mu=\mu}\Big(h_\mu(\Phi)+\int_M p\, d\mu\Big)=h_{\mu_p}(\Phi)+\int_M p\, d\mu_p.
		$$
		Here, the supremum  is taken over all $\Phi$-invariant measures $\mu$ on $M$; we call $P_p$ the \emph{topological pressure} of the potential $p$ with respect to the flow $\Phi$.  
	\end{prop}
	
	\begin{prop}[see Proposition 4.7 in \cite{C}, and also   Proposition 4.5 in  \cite{Bo} for diffeomorphisms]\label{Chernov equ}
		Two equilibrium states $\mu_{p_1}$ and $\mu_{p_2}$ associated to H\"{o}lder potentials $p_1,p_2 \colon M \to \R$ coincide if and only if for any Markov family $\mathcal{R}$,   the functions 
		$$
		G_i\colon x\mapsto \int_0^{\tau(x)} p_i(\Phi^t(x))dt - P_{p_i} \times \tau(x),\qquad i=1,2, %\quad\text{and}\quad G_2(x):=\int_0^{\tau(x)} p_2(\Phi^t(x))dt - P_{p_2} \times \tau(x)
		$$
		are cohomologous on $\mathcal{S}$, where  $\mathcal{S}$ denotes the cross section associated to   $\mathcal{R}$, and $\tau\colon \mathcal{S}\to \R_+$ is the first return time on $\mathcal{S}$. In other words, there exists a H\"{o}lder continuous function $u \colon \mathcal{S} \to \R$ such that  
		$$
		G_2(x)-G_1(x)=u \circ \mathcal{F}(x)-u(x),\quad \forall\, x \in \mathcal{S},
		$$
		where  $\mathcal{F} \colon x \mapsto \Phi^{\tau(x)}(x)$ is  the Poincar\'e map induced by  $\Phi$ on $\mathcal{S}$. 
	\end{prop}

As a direct corollary of Proposition \ref{Chernov equ}, we have:
\begin{corollary}\label{Corolla Chernov}
	If two equilibrium states $\mu_{p_1}$ and $\mu_{p_2}$ associated to H\"{o}lder potentials $p_1,p_2 \colon M \to \R$ coincide, then for any periodic orbit $\mathcal{O}=\{\Phi^t(x)\}_{t \in [0,\mathcal{L}(\mathcal{O})]}$ of period $\mathcal{L}(\mathcal{O})=\mathcal{L}(x)>0$, we have 
	$$
	\int_0^{\mathcal{L}(\mathcal{O})} p_1(\Phi^t(x))dt - P_{p_1}\times  \mathcal{L}(\mathcal{O})=\int_0^{\mathcal{L}(\mathcal{O})} p_2(\Phi^t(x))dt - P_{p_2}\times \mathcal{L}(\mathcal{O}).
	$$
\end{corollary}\quad 

Let us recall that a  \emph{Sinai-Ruelle-Bowen measure}, or \emph{SRB measure} for short, is a $\Phi$-invariant Borel probability measure $\mu$ that is characterized by the property that $\mu$ has absolutely continuous conditional measures on unstable manifolds (see  for instance \cite{Yo} for a reference). For an Anosov flow  on some compact connected  manifold   which preserves a smooth volume, the volume measure is the unique SRB measure.  When volume is not  preserved, SRB measures are the invariant measures most compatible with volume. 

A \emph{measure of maximal entropy}, or \emph{MME} for short, is a $\Phi$-invariant  probability measure $\nu$ which maximizes the metric entropy, i.e., such that $h_{\nu}(\Phi)=h_{\mathrm{top}}(\Phi)$. 

Let us recall that both the SRB measure and MME can be  characterized as equilibrium states (see for instance \cite{C} for a reference):
\begin{prop}[SRB measure/measure of maximal entropy]\label{potential MME-SRB}\quad\\
	\indent $\bullet$ There exists a unique SRB measure   for the flow $\Phi$; it is the unique equilibrium state associated to the geometric potential % \marginpar{I think that for a flow we have to take the derivative, since the time in the Jacobian is not canonical (see for instance p.65 of \cite{C}).}
	$$
	p^u\colon x \mapsto -\frac{d}{dt}|_{t=0} \log J^u_x(t), %\mbox{\color{red} I thought this was just the unstable jacobian itself, not its derivative?}
	$$
	where $J_x^u(t)$ is the Jacobian of the map $D\Phi^t \colon E^u(x) \to E^u(\Phi^t(x))$. Moreover, the pressure for this potential vanishes, i.e., $P_{p^u}=0$. 
	
	$\bullet$ There exists a unique  MME for the flow $\Phi$; it is the unique equilibrium state for the zero potential $p=\underline{0}$. In particular, by the variational principle, the associated pressure is equal to the topological entropy, i.e., $P_{\underline{0}}=h_{\mathrm{top}}(\Phi)$.  
\end{prop}

\subsection{Anosov class and Birkhoff Normal Form}\label{subs anosov class}

In this subsection,  we recall some general notions about obstructions to the regularity of the dynamical foliations of Anosov flows on $3$-manifolds, in particular, the notion of \emph{Anosov class}; we follow the exposition of \cite{HuKa}. 

Let $M$ some $C^\infty$ smooth manifold $M$ of dimension three which supports a $C^k$ volume-preserving  Anosov flow $\Phi=(\Phi^t)_{t \in \R}$, for  some integer $k \geq 3$.   As previously, we denote by $X_\Phi=X:=\frac{d}{dt}|_{t=0} \Phi^t$ the flow vector field, and we let $d\nu :=\imath_X \mathrm{Vol}$.

\begin{defi}[Adapted transverse coordinates, see \cite{HuKa}] \label{def adapted coord}
	For some small $\varepsilon>0$, we say that a $C^1$ map $\Psi\colon M \times (-\varepsilon,\varepsilon)^2 \to M$ defines  $C^k$-adapted transverse coordinates for the flow $\Phi$ if  for each $x \in M$:
	\begin{enumerate}
		\item the map $\Psi_x  \colon (-\varepsilon,\varepsilon)^2 \to M$, $(\xi,\eta)\mapsto \Psi(x,\xi,\eta)$ is a $C^k$ diffeomorphism and the vectors $\partial_\xi \Psi_x$, $\partial_\eta \Psi_x$ are uniformly transverse to $X$; in particular,  $\mathcal{S}_x:=\Psi_x((-\varepsilon,\varepsilon)^2)$ is a uniformly embedded transversal to the flow; 
		\item the maps $\Psi_x (\cdot,0)$, $\Psi_x (0,\cdot)$ are coordinates respectively onto the  stable manifold $\mathcal{W}_{\mathcal{S}_x}^s(x)$ and the unstable manifold $\mathcal{W}_{\mathcal{S}_x}^u(x)$ at $x$ and depend in a $C^1$  fashion on  $x$, when considered  as $C^k$ immersions of $(-\varepsilon,\varepsilon)$ into $M$;
		\item   the $C^1$ foliation $\mathcal{W}_{\mathcal{S}_x}^s$, resp. $\mathcal{W}_{\mathcal{S}_x}^u$ obtaining by restricting the weak stable foliation $\mathcal{W}_\Phi^{cs}$, resp. weak unstable foliation $\mathcal{W}_\Phi^{cu}$ to $\mathcal{S}_x$  is $C^1$-tangent at $x$ to the linear foliation of $\mathcal{S}_x$  by the coordinate lines parallel to the  horizontal axis, resp. vertical axis, in the coordinates provided by $\Psi_x$; 
		\item 	the restriction of $d\nu$ to $\mathcal{S}_x$ satisfies $\Psi_x^{*}(d\nu)=d\xi \wedge d\eta$. 
	\end{enumerate}
\end{defi}

\begin{prop}[Proposition 4.2 in \cite{HuKa}]
	For any $k \geq 3$,  $C^{k-1}$-adapted transverse coordinates exist for a volume-preserving Anosov flow on a closed $3$-manifold. 
\end{prop}

Let $V$ be a vector field on $M$. For each point $x \in M$, the restriction of $V$ to $\mathcal{S}_x$ is projected onto $T\mathcal{S}_x$, and then expressed in the local coordinates as $(\xi,\eta)\mapsto (V_x^1(\xi,\eta),V_x^2(\xi,\eta))$. We denote by $V_x:=V(\cdot,0)$ %\colon \xi\mapsto (V_x^1(\xi),V_x^2(\xi))$ 
the restriction of this vector field to the horizontal axis.   

We let $\tilde V_x$  be the local vector field at $x$ along the stable manifold through $x$ obtained from $V_x$ by pointwise scaling so that in coordinates we have $\tilde V_x\colon z\mapsto (a_{V,x}(z),1)$. We consider the expansion of $a_{V,x}$ near $0$:
$$
a_{V,x}(z)=a^{0}_{V,x}+ a^1_{V,x} z+ z b_{V,x}(z)+o(z^2),
$$
where $a^0_{V,x},a^1_{V,x} \in \R$, and $b_{V,x}$ is continuous and vanishes at $0$. 

In the following, we let $V^*$ be a vector field such that $a^0_{V^*,x}=0$ for any $x \in M$, which corresponds to the case of the weak unstable distribution. Let $\pi_X\colon TM \to TM$ be the fiberwise projection map onto the subbundle of vectors orthogonal to $X$, and for any $(x,t)\in M \times \R$, let $\lambda^*(x,t)\in \R \setminus \{0\}$ be defined as 
$$
\pi_X \circ D \Phi^t(V^*)(\Phi^t(x))=\lambda^*(x,t) V^*(\Phi^t(x)). 
$$

Let us recall the notions of \emph{cocycle} and \emph{coboundary} over a flow. 
\begin{defi}[Cocycle/coboundary]
	A map $C \colon M \times \R \to \R$ is a called a $C^1$ \emph{cocycle} over the flow $\Phi$ if it is of class $C^1$ and satisfies
	$$
	C(x,t+s)=C(x,t)+C(\Phi^t(x),s),\quad \forall\, x \in M,\, \forall\, t,s \in \R. 
	$$
	A $C^1$ cocycle  $B \colon M \times \R\to \R$ over $\Phi$ is a $C^1$ \emph{coboundary} if there exists a $C^1$ function  $u \colon M \to \R$\footnote{By Livsic's theorem, it is sufficient to have a $C^0$ function $u\colon M \to \R$ such that \eqref{cobound} holds.} such that 
	\begin{equation}\label{cobound}
	B (x,t)= u \circ \Phi^t(x)-u(x),\quad \forall\, x \in M,\, \forall\, t \in \R.
	\end{equation}
	The $C^1$\emph{-cohomology class} of a $C^1$ cocycle $C \colon M \times \R \to \R$  over $\Phi$ is  the image of $C$ in the group of $C^1$ cocycles over $\Phi$  modulo the $C^1$ coboundaries. 
\end{defi}

\begin{lemma}[Anosov cocycle/class, see Lemma 5.1, Proposition 5.3 in \cite{HuKa}]
	For any $(x,t) \in M \times \R$, we denote by $\tilde V^*_{x,t}$ the rescaled image of $\tilde V_x^*$ by the Poincar\'e map of $\Phi$ from $\mathcal{S}_x$ to $\mathcal{S}_{\Phi^t(x)}$, so that $\tilde V^*_{x,t}\colon z\mapsto (a_{V^*,x,t}(z),1)$. For $|z|\ll 1$, we have
	$$
	a_{V^*,x,t}(z)=(\lambda^*)^{-1} a^1_{V^*,x} z +(\lambda^*)^{-1}  z b_{V^*,x}(\lambda^* z)+ A_\Phi(x,t)z^2  %\alpha_0 B_\Phi^\Psi(x,t)z^2
	+o(z^2),
	$$
	where $\lambda^*=\lambda^*(x,t) \neq 0$, and $A_\Phi(x,t) \in \R$. 
	Then, the map $A_\Phi \colon M \times \R \to \R$ is a $C^1$ \emph{cocycle} over the flow $\Phi$, 
	called the \emph{Anosov cocycle}. 
	
	The \emph{Anosov class} $[A_\Phi]$ is defined as the $C^1$-cohomology class of $A_\Phi$; it is independent of the choice of the Riemannian  metric on $TM$ and $C^k$-adapted transverse coordinates. 
\end{lemma}

Let us now consider the case where $x$ is a point in some periodic orbit $\mathcal{O}$ for $\Phi$ of period $\mathcal{L}(\mathcal{O})=\mathcal{L}(x)>0$. We let $\Psi$ be  $C^k$-transverse coordinates for  $\Phi$ according to Definition \ref{def adapted coord},  and we let $\mathcal{S}_x$ be the associated transverse section    to $\Phi$ at $x$,  endowed with local coordinates $(\xi,\eta)$, the point $x$ being identified with the origin $(0,0)$.  The Poincar\'e  map $\mathcal{F}$ induced by the flow $\Phi$ on $\mathcal{S}_x$  is a  local   diffeomorphism defined in a neighbourhood of $(0,0)$, and  preserves the volume $d\nu=\imath_X  \mathrm{Vol}$. It has a saddle fixed point at $(0,0)$ with eigenvalues $0<\lambda<1 <\lambda^{-1}$, and is written in coordinates as 
\begin{equation*}
\mathcal{F}\colon (\xi,\eta)\mapsto (\lambda \xi + \mathcal{F}_1(\xi,\eta),\lambda^{-1} \eta + \mathcal{F}_2(\xi,\eta)). 
\end{equation*}

 In fact,  assuming  that $\mathcal{S}_x$ is chosen sufficiently small, then for $k=\infty$, by a result of Sternberg \cite{Ste}, there exists a $C^\infty$ volume-preserving change of coordinates
$
R_0 \colon
\mathcal{S}_x  \to \R^2$ which conjugates $\mathcal{F}$ to its \emph{Birkhoff Normal Form} 
$N=R \circ \mathcal{F}\circ R^{-1}$:
\begin{align*}
N=N_\Delta\colon (\xi,\eta)\mapsto(\Delta(\xi\eta)\cdot\xi,\Delta(\xi\eta)^{-1}\cdot\eta),
\end{align*}
for some (unique) function
$$
\Delta\colon z \mapsto a_0+a_1 z+  a_2 z^2+\dots,\quad \text{with }a_0=\lambda\neq 0.
$$
The numbers $(a_k)_{k \geq 0}=(a_k(\mathcal{O}))_{k \geq 0}$ are called the \textit{Birkhoff
	invariants} or \textit{coefficients} at  $\mathcal{O}$ of $\mathcal{F}$.

This has later been generalized to the case of finite regularity $k \geq 3$ in several works  (see for instance  \cite{GST} and \cite{DGG}). In particular, there exists a $C^k$ change of coordinates  $R \colon \mathcal{S}_x \to \R^2$ under which $\mathcal{F}$ takes the form
\begin{equation*}
	\mathcal{F}\colon (\xi,\eta)\mapsto (\lambda \xi + a_1 \xi^2 \eta + o(\xi^2 \eta),\lambda^{-1} \eta - a_1 \xi \eta^2 + o(\xi \eta^2)). 
\end{equation*}

By a direct calculation, we have the following identity between the Anosov cocycle and the first  Birkhoff invariant at  $\mathcal{O}$: (see formula (16) in Hurder-Katok \cite{HuKa})
\begin{equation}\label{anosov birkh}
A_\Phi(x,\mathcal{L}(x))=\frac 12 \lambda \partial_{\xi \eta\eta} \mathcal{F}_2 (0,0)=-\frac 12 \lambda^{-1}  \partial_{\eta\xi\xi}  \mathcal{F}_1(0,0)=- \lambda^{-1}  a_1.
\end{equation}

The following result of Hurder-Katok says that the  Anosov class corresponds to certain  obstructions to the smoothness of the weak stable/weak unstable distributions. Moreover, by Livsic's Theorem, the Anosov class $[A_\Phi]$ vanishes if and only if  $A_\Phi(x,\mathcal{L}(x))=0$ for any periodic point $x$ of period $\mathcal{L}(x)>0$.   In other words, it is sufficient to consider what happens at periodic points,  and in view of \eqref{anosov birkh}, the periodic obstructions can be characterized in terms of the first Birkhoff invariant. 
\begin{theorem}[Theorem 3.4, Corollary 3.5, Proposition 5.5 in \cite{HuKa}]\label{theorem HK}
	Let us assume that $k \geq 5$. The following properties are equivalent:
	\begin{itemize}
		\item the Anosov class $[A_\Phi]$ vanishes;
		\item for any periodic orbit $\mathcal{O}$, the first Birkhoff invariant at  $\mathcal{O}$ vanishes;
		\item the weak stable/weak unstable distributions $E_\Phi^{cs}$/$E_\Phi^{cu}$ are $C^{k-3}$.
	\end{itemize} 
\end{theorem}

Besides, by the work of Ghys, we know that high regularity of the dynamical distributions implies that the flow is orbit equivalent to an algebraic model:
\begin{theorem}[Th\'eor\`eme 4.6 in \cite{G2}]\label{theoreme ghys}
	Let  $\Phi\colon M \to M$ be a $C^k$ Anosov flow on some $3$-manifold $M$, for some integer $k \geq 2$.  If $\mathcal{W}_\Phi^{cu}$ and $\mathcal{W}_\Phi^{cs}$ are of class $C^{1,1}$, then $\Phi$ is $C^k$-orbit equivalent to an algebraic flow.  %quasi-Fuchsian flow or to the suspension of a diffeomorphism of the $2$-torus $\T^2=\R^2/\Z^2$. 
\end{theorem}

\section{Entropy rigidity for conservative Anosov   flows on $3$-manifolds}

\subsection{Periodic Lyapunov exponents when SRB=MME}\label{subs periodic exp}

Let $M$ be a $C^\infty$ smooth compact Riemannian manifold of dimension $3$. Given  an integer $k \geq 2$, we consider the restriction $\Phi|_{\Lambda}$ of some $C^k$ Axiom A flow $\Phi\colon M \to M$ to some basic set $\Lambda\subset M$ (or a $C^k$ Anosov flow $\Phi$). Moreover, we assume that $\Phi_* \mu=\mu$ for some smooth volume measure $\mu$ (in particular, in the case where  $\Phi$ is Anosov, it is transitive).  %that is absolutely continuous with respect to the volume. 
Equivalently\footnote{See for instance Theorem 4.14 in \cite{Bo}.}, for any periodic orbit $\mathcal{O}=\{\Phi^t(x)\}_{t \in [0,\mathcal{L}(\mathcal{O})]}$ of period $\mathcal{L}(\mathcal{O})=\mathcal{L}(x)>0$, the map $D_x\Phi^{\mathcal{L}(x)}\colon T_x M \to T_x M$ has determinant one. In particular, the Lyapunov exponent of  the orbit $\mathcal{O}$ for the flow $\Phi$ is equal to
$$
\mathrm{LE}(\mathcal{O})=\mathrm{LE}(x)=\frac{1}{\mathcal{L}(x)}   \log J^u_x(\mathcal{L}(x)).
$$
Moreover, by Proposition \ref{potential MME-SRB},    there exist a unique SRB measure for $\Phi|_{\Lambda}$ (when $\Phi$ is Anosov, this SRB measure is equal to $\mu$) and a unique measure of maximal entropy (MME).  
Combining Corollary \ref{Corolla Chernov} and Proposition \ref{potential MME-SRB}, we deduce:
\begin{prop}\label{prop equal le}
	If the SRB measure is equal to the MME, then the Lyapunov exponents of periodic orbits are constant, i.e., for any periodic orbit $\mathcal{O}$, we have 
	$$
	\mathrm{LE}(\mathcal{O}) =h_{\mathrm{top}}(\Phi). 
	$$
	Equivalently,  for any $x \in \mathcal{O}$, and if $\mathcal{L}(\mathcal{O})=\mathcal{L}(x)>0$ is  the period of $\mathcal{O}$, it holds 
	\begin{equation}\label{logunsta}
	J^u_x(\mathcal{L}(x))=e^{h_{\mathrm{top}}(\Phi)\mathcal{L}(x)}.  
	\end{equation}
\end{prop}

\begin{proof}
	By Proposition \ref{potential MME-SRB},  the SRB measure and the MME  are respectively associated to the  potentials $p^u$ and $\underline{0}$, and to the pressures $0$ and $h_{\mathrm{top}}(\Phi)$. We deduce from Corollary \ref{Corolla Chernov} that for any  periodic orbit $\mathcal{O}=\{\Phi^t(x)\}_{t \in [0,\mathcal{L}(\mathcal{O})]}$ of period $\mathcal{L}(\mathcal{O})=\mathcal{L}(x)>0$, it holds 
	$
	\log J^u_x(\mathcal{L}(x))  =h_{\mathrm{top}}(\Phi)   \mathcal{L}(x)
	$, which concludes. 
\end{proof}

In Appendix \ref{appendix bowen}, we outline another approach for topologically mixing Anosov flows, based on the properties of the  Bowen-Margulis measure.

\subsection{Expansion of the Lyapunov exponents of periodic orbits in a horseshoe with prescribed combinatorics}
\label{subsec:horseshoe-homoclinic}

Let $M$  and $\Phi$ be as in Subsection \ref{subs periodic exp}.  
The goal in this subsection is to derive  asymptotics on the Lyapunov exponents of certain periodic orbits in the horseshoe associated to some homoclinic intersection between the weak stable and weak unstable manifolds of some reference periodic point. For this, we select a sequence of periodic orbits accumulating the given periodic point with a prescribed combinatorics. 

%Let $\Lambda$ be an attractor of $\Phi$ ($\Lambda=M$ when $\Phi$ is a transitive Anosov flow). It is encoded by a   subshift of finite type $\sigma_A \colon \Sigma_A \to \Sigma_A$ for some matrix $A=(A_{i,j})_{1 \leq i,j\leq n} \in \mathfrak{M}_n(\R)$, with $A_{i,j} \in \{0,1\}$ for all $1\leq i,j \leq n$, and 
%$$
%\Sigma_A:=\Big\{(x_k)_{k=-\infty}^{+\infty}\in \{1,\dots,n\}^\Z: A_{x_k,x_{k+1}}=1\Big\}. 
%$$
%More precisely, for a certain roof function $\tau \colon \Sigma_A \to \R_+$,  for $\Lambda(A,\tau):=(\Sigma_A \times \R)/\sim$, where $(w,\tau(w))\sim (\sigma_A (w),0)$, $\forall\, w \in \Sigma_A$,  
%the suspension flow $\mathscr{T}=(\mathscr{T}^t)_{t \in \R}$ on  $\Lambda(A,\tau)$ is semi-conjugated to $\Phi$ by a continuous surjection $\rho \colon \Lambda(A,\tau) \to \Lambda$: 
%$$
%\rho \circ \mathscr{T}^t = \Phi^t \circ \rho,\qquad \forall\, t \in \R.
%$$

In the following, we take a Markov family $\mathcal{R}$ for $\Phi$ associated to some cross section $\mathcal{S}$, and 
we denote by    $\mathcal{F}$ the Poincar\'e map induced by $\Phi$ on $\mathcal{S}$. Let $x \in \mathcal{S}$ be  a point in some periodic orbit $\mathcal{O}=\{\Phi^t(x)\}_{t \in [0,\mathcal{L}(\mathcal{O})]}$ of period $\mathcal{L}(\mathcal{O})=\mathcal{L}(x)>0$.  
%We assume that $x$ is encoded by a word $w=(\dots w_0w_0w_0\dots) \in \Sigma_A$. % for some letter $w_0$. 
We consider a point $x_\infty\in \mathcal{S}$ such that $x_\infty \in \mathcal{W}_\Phi^{cs}(x)\pitchfork \mathcal{W}_\Phi^{cu}(x)$. In particular, the orbit $h_\infty$ of $x_\infty$ is homoclinic to the periodic orbit $\mathcal{O}$.  It is well-known that this tranverse homoclinic intersection generates a horseshoe which admits a symbolic coding (see for instance \cite[Theorem 6.5.5]{HaKa}). Let $\mathcal{S}_0\subset \mathcal{S}$  be a small neighbourhood of the periodic point $x$ encoded by the symbol $0$, and let $\mathcal{S}_1\subset  \mathcal{S}$ be a small neighbourhood of the homoclinic point $x_\infty$ encoded by the symbol $1$.  After possibly replacing $\mathcal{F}$ with some iterate $\mathcal{F}^p$, $p \geq 1$, the symbolic coding associated to $x_\infty$ is
\begin{equation*}
x_\infty\, \longleftrightarrow\,
\dots 000\underset{\substack{\uparrow}}{1}000\dots
\end{equation*} 
Let us select the sequence $(h_n)_{n \geq 0}$ of periodic orbits in the horseshoe whose (periodic) symbolic coding is given by
$$
h_n\, \longleftrightarrow\, \dots\vert \underbrace{0\dots 0}_{n+1}1\vert \underbrace{0\dots 0}_{n+1}1\vert \underbrace{0\dots 0}_{n+1}1\vert \dots%\underbrace{w_0\dots w_0}_{n}).
$$
In particular, for each $n \geq 0$, $h_n$ is periodic for the Poincar\'e map $\mathcal{F}$,  of period $n+2$. 

\begin{figure}[H]
	\begin{center}
		\includegraphics [width=9.5cm]{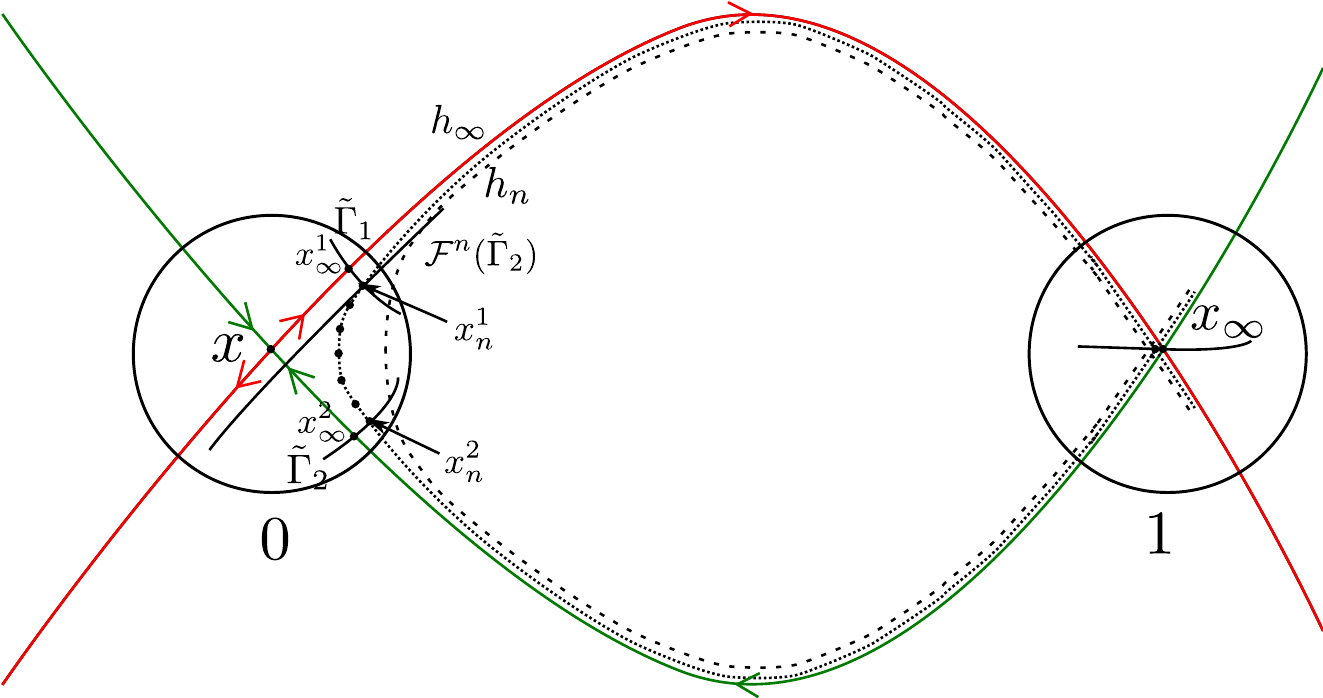}
		\caption{Selection of the periodic orbits $h_n$.}\label{picture horseshoe}
	\end{center}
\end{figure}

After going to  a chart, we endow $\mathcal{S}_0$ with local coordinates. 
 We denote by $x_\infty^1$, resp. $x_\infty^2$, the coordinates of the point of $h_\infty$ in $\mathcal{S}_0$ encoded by the symbolic sequence 
\begin{equation*}
  x_\infty^1\, \longleftrightarrow\,
  \dots 000\underset{\substack{\uparrow}}{0}10000\dots,\quad  \text{resp. }x_\infty^2\, \longleftrightarrow\,
  \dots 00001\underset{\substack{\uparrow}}{0}000\dots
\end{equation*} 
    %$x_n^1$, resp. $x_n^2$ the coordinates of the point in the orbit $h_n$ that is labelled by the last, resp. first symbol $w_0$ before, resp. after the word $v_0$. 
Similarly, for any integer $n \geq 0$, we denote by $x_n^1$, resp. $x_n^2$, the coordinates of  its periodic approximation in $h_n$, encoded by the symbolic sequence 
\begin{equation*}
x_n^1\, \longleftrightarrow\,
\dots \vert 0\dots 0 \underset{\substack{\uparrow}}{0}1\vert 0\dots 0 1\vert \dots ,\quad  \text{resp. }x_n^2\, \longleftrightarrow\,
\dots \vert 0\dots 0 1\vert \underset{\substack{\uparrow}}{0}0\dots 0 1\vert \dots
\end{equation*}     
Note that the points $x_\infty^1,x_n^1$, resp. $x_\infty^2,x_n^2$ share the same symbolic past and future for $n$ steps, hence by hyperbolicity, they are exponentially close in phase space  for $n \gg 1$   large. See Figure \ref{picture horseshoe} for an illustration.

 % In particular, we have $x_n^1=\mathcal{F}^{n}(x_n^2)$. % , while the point $x_n$ is periodic, % of period $n$, 
%for each $n \geq 1$. 

 The point $x$ is a saddle fixed point under $\mathcal{F}$, with eigenvalues $0<\lambda<1<\lambda^{-1}$.  
The restriction of $\mathcal{F}$ to $\mathcal{S}_0$ is a  local volume-preserving diffeomorphism. As recalled in Subsection \ref{subs anosov class}, if $\mathcal{S}_0$ is chosen sufficiently small, then there exists   a $C^k$  change of coordinates
$
R\colon
\mathcal{S}_0  \to \R^2$ under which $\mathcal{F}$ takes the form 
\begin{equation*}
	N=N_\Delta \colon (\xi,\eta)\mapsto (\Delta(\xi\eta) \xi,\Delta(\xi\eta)^{-1} \eta), 
\end{equation*}
for some  $C^k$ function $\Delta \colon \R \to \R$ such that 
$
\Delta(z)= \lambda + a_1 z + o(z)$ for $|z|\ll 1$, where $a_1=a_1(\mathcal{O})$ is the first Birkhoff invariant at $\mathcal{O}$ of $\mathcal{F}$.  For $k=\infty$, $\Delta$ is $C^\infty$, and $N$ is the Birkhoff Normal Form of $\mathcal{F}$. 

For simplicity, in the following, we only detail the case where $k=\infty$, but the case of finite regularity $k \geq 5$ is handled similarly. 
%\begin{lemma}
%	The points $x_n$ belong to a smooth arc $\Gamma\subset S$. 
%\end{lemma}

%\begin{proof}
%	Let us consider a subsection $\tilde S\subset S$ such that for $n \geq 1$ sufficiently large, $x_n$ is the unique point in the orbit $h_n$ in $S'$. Then, the points $x_n$ can be characterized in terms of the time-reversal poperty. Indeed, if $\tilde{\mathcal{F}}$ is the Poincar\'e map induced by $\Phi$ on $\tilde S$ and $\tilde{\mathcal{F}}^{-1}$  the Poincar\'e map induced by the flow $\Phi^{-1}$ on $\tilde S$, then  the set 
%	$$
%	\mathcal{S}:=\{z \in \tilde S: \tilde{\mathcal{F}}(z)=\tilde{\mathcal{F}}^{-1}(z)\}. 
%	$$ 
%	The set $\mathcal{S}$ is smooth and contains the periodic points $x_n$ for $n$ large enough. 
%\end{proof}

\begin{lemma}\label{choice of R}
	The conjugacy $R$ can be chosen in such a way that for  all $n \gg 1$, 
	\begin{equation*}
	R(x_n^1)=(\eta_n,\xi_n)\in \Gamma_1,\qquad%\in \{\xi=\eta\},
	R(x_n^2)=(\xi_n,\eta_n)\in \Gamma_2,%\in \{\xi=\eta\},\quad \forall \, n \geq 1.
	\end{equation*}
	where  $\Gamma_1,\Gamma_2$ 
	are two smooth arcs which  are  mirror images of each other   under the reflection with respect to the first bissectrix $\{\xi=\eta\}$. 
\end{lemma}

\begin{proof}
	Let $\tilde R$ be any $C^\infty$ volume-preserving map such that $N=\tilde R \circ \mathcal{F}\circ (\tilde R)^{-1}$.  	
	%For any integer $k \geq 0$, we set $\tilde R^k_+:=N^{k} \circ \tilde R\circ \mathcal{F}^{-k}$ and $\tilde R^k_-:=N^{-k}\circ \tilde R\circ \mathcal{F}^k$. Given any point $z \in \mathcal{S}$ such that both $z$ and $\mathcal{F}^{-k}(z)$ belong to the neighbourhood of $x \in \mathcal{S}$ where $\tilde R$ is defined, then by definition, $\tilde R^k_+$ coincides with $R$, and similarly for 
	%$\tilde R^k_-$. Those definitions allow us to extend the domain of definition of $\tilde R$ along the unstable manifold of $x$ by $\tilde R^k_+$ and along the stable manifold of $x$ by $\tilde R^k_-$.  Doing this repeatedly until we reach regions close to the points $x_\infty^1,x_\infty^2$ of $h_\infty$ defined previously, %encoded by the symbol $w_0$ right before and after the word $v_0$, 
	%we get a map which we still denote by $\tilde R$. 
	Inside $\mathcal{S}_0$, we have a foliation $\mathscr{H}$ by curves along which the motion happens, which corresponds to the preimage of the foliation by  the hyperbolas $\{\xi\eta=\mathrm{cst}\}$ under $\tilde R$. Let us consider the square of the Poincar\'e map $\mathcal{F}$ restricted to a small neighbourhood $\mathcal{S}_\infty^1\subset \mathcal{S}_0$ of $x_\infty^1$. It follows from the definition of $x_\infty^1,x_\infty^2$ that $\mathcal{F}^2$ maps $\mathcal{S}_\infty^1\subset \mathcal{S}_0$ to a small neighbourhood $\mathcal{S}_\infty^2\subset \mathcal{S}_0$ of $x_\infty^2$. The leaf of $\mathscr{H}$ through $x_\infty^1$ coincides with the local unstable leaf $\mathcal{W}_{\mathcal{F},\mathrm{loc}}^u(x_\infty^1)\subset \mathcal{W}_{\mathcal{F}}^u(x)$, while the leaf of $\mathscr{H}$ through $x_\infty^2$ coincides with the local stable leaf $\mathcal{W}_{\mathcal{F},\mathrm{loc}}^s(x_\infty^2)\subset \mathcal{W}_{\mathcal{F}}^s(x)$. In particular, locally, the leaves of the foliation $\mathscr{H}\cap  \mathcal{S}_\infty^1$ are close to the unstable leaf $\mathcal{W}_{\mathcal{F},\mathrm{loc}}^u(x_\infty^1)$, while the leaves of  $\mathscr{H}\cap  \mathcal{S}_\infty^2$ are close to the stable leaf $\mathcal{W}_{\mathcal{F},\mathrm{loc}}^s(x_\infty^2)$. Besides, due to the presence of the homoclinic intersection between $\mathcal{W}_{\mathcal{F}}^u(x)$ and $\mathcal{W}_{\mathcal{F}}^s(x)$, the image under   $\mathcal{F}^2$ of $\mathcal{W}_{\mathcal{F},\mathrm{loc}}^u(x_\infty^1)$ intersects $\mathcal{W}_{\mathcal{F},\mathrm{loc}}^s(x_\infty^2)$ transversally. We conclude that the foliation $\mathcal{F}^2(\mathscr{H}\cap \mathcal{S}_\infty^1)=\{\mathcal{F}^2(\mathscr{H}_t^1)\}_{t}$ is transverse to the foliation $\mathscr{H}\cap \mathcal{S}_\infty^2=\{\mathscr{H}_t^2\}_{t}$ provided that $\mathcal{S}_\infty^1,\mathcal{S}_\infty^2$ are chosen sufficiently small, where for $|t|\ll 1$,  $\mathscr{H}_t^1$ denotes the leaf of the foliation $\mathscr{H}\cap \mathcal{S}_\infty^1$ coming from the hyperbola $\{\xi\eta=t\}$, and $\mathscr{H}_t^2$ denotes  the leaf of the foliation $\mathscr{H}\cap \mathcal{S}_\infty^2$ coming from the  same hyperbola $\{\xi\eta=t\}$. 
	%Let $S'$ be a section transverse to the flow at a point of $h_\infty$ in the region encoded by $v_0$. In this region, sufficiently close to the homoclinic orbit, the continuations of those two foliations intersect transversely (as the leaves are close respectively to the stable/unstable manifold of $x$); 
	Since the foliation $\mathscr{H}$  is smooth (it is the image under $\tilde R$ of the foliation $\{\xi\eta=\mathrm{cst}\}$), the locus of intersection $\{\mathcal{F}^2(\mathscr{H}_t^1)\cap \mathscr{H}_t^2\}_{|t| \ll 1}$ is a $C^\infty$ curve %.  Let %$\Gamma_1:=\{(\xi,\gamma_1(\xi))\}$, resp.
	$\tilde \Gamma_2$ %, be the images of this smooth  curve on $S$ near 
	containing the point $x_\infty^2$. Similarly, we denote by $\tilde \Gamma_1$ the locus of intersection of  $\mathcal{F}^{-2}(\mathscr{H}_t^2)$ and  $\mathscr{H}_t^1$  near the point $x_\infty^1$ (see Figure \ref{picture horseshoe}).

	Let us denote by $\hat \Gamma_1$ and $\hat \Gamma_2$ the respective images of the arcs  $\tilde \Gamma_1$ and $\tilde \Gamma_2$ under $\tilde R$. 
	For any $C^\infty$ function $D$, the map $N_{D}\colon (\xi,\eta)\mapsto(D(\xi\eta)\cdot\xi,D(\xi\eta)^{-1}\cdot\eta)$ commutes with $N$, hence $N_{D} \circ \tilde R$ also conjugates $\mathcal{F}$ with its Birkhoff Normal Form. By replacing $\tilde R$ with $N_{D} \circ \tilde R$ for some suitable function $D$, without loss of generality, we may assume that locally, $\hat \Gamma_1=\{(\xi,\hat \gamma_{1}(\xi))\}_\xi$, resp. $\hat \Gamma_2=\{(\xi,\hat \gamma_{2}(\xi))\}_\xi$, is a graph near $\tilde R(x_\infty^1)=:(0,\tilde \xi_\infty^1)$, resp. $\tilde R(x_\infty^2)=:(\tilde \xi_\infty^2,0)$.  Let us look for $D$ such that $N_{D}(\hat \Gamma_1)=\{(D(\xi\hat \gamma_1(\xi))\xi,D(\xi\hat \gamma_1(\xi))^{-1}\hat \gamma_1(\xi))\}_\xi$ and  $N_{D}(\hat \Gamma_2)=\{(D(\xi\hat  \gamma_2(\xi))\xi,D(\xi\hat \gamma_2(\xi))^{-1}\hat \gamma_2(\xi))\}_\xi$ are mirror images of each other under  the involution $\mathcal{I}_0\colon (\xi,\eta)\mapsto (\eta,\xi)$. As the latter preserves hyperbolas $\{\xi\eta=\mathrm{cst}\}$, this happens if and only if 
	$$
	D(\xi\hat \gamma_1(\xi))^{-1}\hat\gamma_1(\xi)=D(\hat  \xi\hat \gamma_2(\hat \xi))\hat{\xi},
	$$
	where    $\hat \xi=\hat\xi(\xi)$ is the unique number such that $\hat\xi \hat \gamma_2(\hat \xi)=\xi\hat \gamma_1(\xi)$. Note that for $|\xi|$ small, i.e., for  $(\xi,\hat \gamma_1(\xi))$ close to $\tilde R(x^1_\infty)=(0,\tilde \xi_\infty^1)$, the existence of $\hat \xi(\xi)$ is guaranteed by the implicit function theorem. Indeed, by the transverse intersection between the stable  and the unstable manifolds at $x_\infty^2$, we have  $\hat \gamma_2'(\tilde \xi_\infty^2)\neq 0$;  moreover, the map $\xi\mapsto \hat \xi(\xi)$ is smooth. The previous equation thus yields 
	$$
	D(\xi\hat \gamma_1(\xi))=\sqrt{\frac{\hat \gamma_1(\xi)}{\hat\xi(\xi)}}.
	$$
	As $\lim_{\xi\to 0}\hat \gamma_1(\xi)=\tilde \xi_\infty^1\neq 0$ and $\lim_{\xi\to 0}\hat \xi(\xi)=\tilde \xi_\infty^2\neq 0$, and since $\hat \gamma_1'(0)\neq 0$, the function $\xi \mapsto D(\xi)$ it defines is smooth near $0$, and the associated change of coordinates $R:=N_{D} \circ \tilde R$ satisfies the desired conditions.

	\begin{figure}[H]
		\begin{center}
			\includegraphics [width=12.5cm]{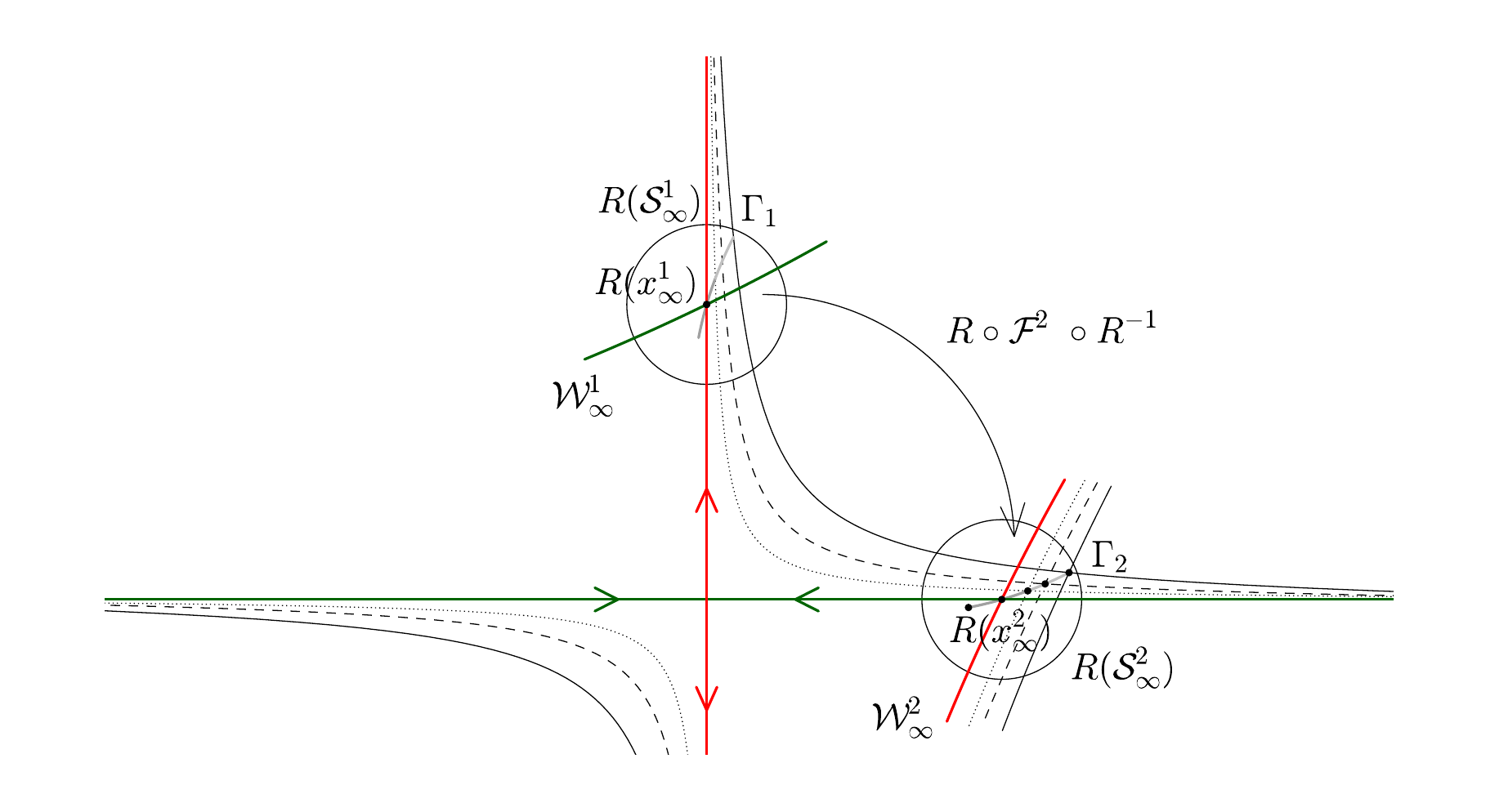}
			\caption{Periodic orbits in Birkhoff coordinates.}\label{birkhoff picture}
		\end{center}
	\end{figure}
	
	Let   $\Gamma_1$ and $\Gamma_2$ be the respective images of $\tilde \Gamma_1$ and $\tilde \Gamma_2$ under $R$. 
	Note that for any large integer $n \gg 1$, the periodic orbit $h_n$ has ``cyclicity'' one, hence the points $x_n^1,x_n^2\in h_n$ belong to the same curve $R^{-1}(\{\xi\eta=t_n\})$, for some $t_n \in \R$. Thus, 
	$$
	(\xi_n,\eta_n):=R(x_n^2)=R(\mathcal{F}^{2}(x_n^1)) \in R\big( \mathcal{F}^2(\mathscr{H}_{t_n}^1)\cap \mathscr{H}_{t_n}^2\big)=\Gamma_2.
	$$ 
	Similarly, $R(x_n^1)\in\Gamma_1$. As $\Gamma_1,\Gamma_2$ are mirror images of each other under $(\xi,\eta)\mapsto (\eta,\xi)$, and $R(x_n^1),R(x_n^2)\in \{\xi\eta=t_n\}$, with $t_n=\xi_n \eta_n$, we have $R(x_n^1)=(\eta_n,\xi_n)$. 
	%	For each $n \geq 1$, let us denote by $(\tilde \xi_n,\tilde \eta_n)$ the coordinates of $\tilde R(x_n)$. We want to choose $\tilde \Delta$ in such a way that $N_{\tilde \Delta} \circ \tilde R$ has the desired properties. For any $n \geq 1$, we set $\zeta_n:=\sqrt{\tilde \xi_n \tilde \eta_n}$. Then, we want that 
	%	$$
	%	\tilde \Delta(\zeta_n^2)=\frac{\tilde \eta_n}{\tilde \xi_n}. 
	%	$$
\end{proof}

In the following, we fix a conjugacy map $R$ as given by  Lemma \ref{choice of R}.  We denote by $\mathcal{W}^1_\infty\subset R(\mathcal{S}_\infty^1)$ the arc of the stable manifold through $R(x_\infty^1)=:(0,\xi_\infty)$, and we let $\mathcal{W}^2_\infty\subset R(\mathcal{S}_\infty^2)$ be  the arc of the unstable manifold through $R(x_\infty^2)=:(\xi_\infty,0)$.  See Figure \ref{birkhoff picture} for an illustration. 

\begin{lemma}\label{transverse}
The arcs $\Gamma_1$, $\mathcal{W}^1_\infty$ and $\{(0,\xi_\infty+\eta):|\eta|\ll 1\}$ are pairwise transverse at $(0,\xi_\infty)$. 
\end{lemma}

\begin{proof}
	The last two arcs are transverse because they are  pieces respectively of the stable and of the unstable manifold of the origin. In the following, we assume that $\Gamma_1=\{(\eta, \gamma(\eta)):|\eta| \ll 1\}$ is the graph of some smooth function $\gamma$ and show that it is transverse to $\mathcal{W}^1_\infty$ at $(0,\xi_\infty)$. The other case is handled similarly. 
	
We use the same notation as in	Lemma \ref{choice of R}. Recall that the arc $\Gamma_1$ is the locus of intersection of $\{R \circ \mathcal{F}^{-2}(\mathscr{H}_t^2)\}_{t}$ and $\{R(\mathscr{H}_t^1)\}_{t}$. Without loss of generality, after going to some chart, we may assume that for $|t|$ small, the foliations $\{R \circ \mathcal{F}^{-2}(\mathscr{H}_t^2)\}_{t}$ and $\{R(\mathscr{H}_t^1)\}_{t}$  are given by
\begin{align*}
R(\mathscr{H}_t^1)&=\{\eta=t\},\\
R \circ \mathcal{F}^{-2}(\mathscr{H}_t^2)&=\{(\eta,f(t,\eta)):|\eta|\ll 1\},
\end{align*}
for some smooth function $f(\cdot,\cdot)$, 
so that $\gamma(\eta)=f(\eta,\eta)$, for  $|\eta|$ small. On the one hand, the arc $\mathcal{W}_\infty^1$ coincides with  the leaf $R \circ \mathcal{F}^{-2}(\mathscr{H}_0^2)=\{(\eta,f(0,\eta)):|\eta|\ll 1\}$, whose tangent space at $(0,\xi_\infty)$ is equal to 
\begin{equation}\label{defi v un}
\mathrm{Span}(1,\partial_2f(0,0)).
\end{equation}
On the other hand, the tangent space to $\Gamma_1=\{(\eta,f(\eta,\eta)):|\eta| \ll 1\}$ at $(0,\xi_\infty)$ is equal to 
$$
\mathrm{Span}(1,\partial_1 f(0,0)+\partial_2f(0,0)).
$$ 
As $\partial_1 f(0,0)\neq 0$, we conclude that $\mathcal{W}^1_\infty$ and $\Gamma_1$ are transverse at $(0,\xi_\infty)$. 
%Similarly, let us denote by  $\mathcal{W}^-\subset \mathcal{S}_\infty^2$ the arc of the unstable manifold through $x_\infty^2$. Let us consider the images $R(\mathcal{W}^-),R(\mathcal{W}^+),R(\Gamma_1),R(\Gamma_2)$ under $R$. By our choice of $R$, the points $x_\infty^1$ and $x_\infty^2$ are symmetric with respect to the first bissectrix $\{\xi=\eta\}$. As the symmetry $(\xi,\eta)\mapsto (\eta,\xi)$ conjugates $N$ to $N^{-1}$, it maps $R(\mathcal{W}^+)$ to $R(\mathcal{W}^-)$. Moreover, by our choice of the conjugacy $R$, the arcs $R(\Gamma_1)$ and $R(\Gamma_2)$  are symmetric with respect to   $\{\xi=\eta\}$. If $\mathcal{W}^+$ and $\Gamma_1$ are not transverse at $x_\infty^1$, then the angle  between $R(\mathcal{W}^+)$ and $R(\Gamma_1)$ at $R(x_\infty^1)$ vanishes.  By the symmetries we have just discussed, this angle is also equal to the angle between $R(\mathcal{W}^-)$ and $R(\Gamma_2)$ at $R(x_\infty^2)$. By transversality between the stable and unstable manifolds at $R(x_\infty^2)$, this implies that the angle between the horizontal direction and $R(\Gamma_2)$ at $R(x_\infty^2)$ is non-zero. But these two manifolds are the respective images of $R(\mathcal{W}^+)$ and $R(\Gamma_1)$ under the image $R \circ \mathcal{F} \circ R^{-1}$ of the image of the Poincar\'e map in those coordinates, which were assumed to be tangent at $R(x_\infty^1)$, a contradiction. 
\end{proof}

In the following,  we denote by $\gamma_0+\gamma_1 \eta+\gamma_2\eta^2+\dots$ the Taylor expansion at $0$ of the smooth function $\gamma$   whose graph is equal to $\Gamma_1=\{(\eta, \gamma(\eta)):|\eta| \ll 1\}$. Note that $\gamma_0=\xi_\infty$.  Moreover, for any integer $n \geq 0$, we abbreviate $x_n^1=x_n$, and we  denote by 
%$$
$\mathcal{L}_n=\mathcal{L}(x_n)>0$ %:=\tau(\mathcal{F}^{n-1}(x_n))+\dots+\tau(\mathcal{F}(x_n))+\tau(x_n)
%$$  
the period of $x_n$ for $\Phi$.  %, i.e., %the smallest positive number $T>0$ such that $\Phi^T(x_n)=x_n$. 
We also denote by $\mathcal{L}^{(0)}=\mathcal{L}(x)>0$ the period of $x$, and we set   $\lambda:=e^{-h_{\mathrm{top}}(\Phi)\mathcal{L}^{(0)}}\in (0,1)$.

\begin{lemma}
	We have the following asymptotic expansion for the periods $(\mathcal{L}_n)_{n \geq 0}$:
	$$
	\mathcal{L}_n=n \mathcal{L}^{(0)}+\mathcal{L}^{(1)} +  O(\lambda^n),
	$$
	for some constant $\mathcal{L}^{(1)}  \in \R$.   %and with $\lambda:=\lambda_0^{\mathcal{L}_0} \in (0,1)$. 
\end{lemma}

\begin{proof}
	We refer the reader to  \cite[Lemma 4.2]{FMT} for a detailed proof. Similar estimates in the case of dispersing billiards can also be found   in \cite[Section 4]{BDKL}.
	
	Let us outline some of the steps of the proof.  
	The computations are carried  out by linearizing the dynamics in a neighbourhood of the saddle fixed point $x$. More precisely, we consider a cross section $\mathcal{S}$ as before and endow it with local coordinates near $x \in \mathcal{S}$. We denote by $\mathcal{F}$ the Poincar\'e map induced by the flow $\Phi$ on $\mathcal{S}$, and by $\tau \colon \mathcal{S} \to \R_+$  the first return time on $\mathcal{S}$.  
	For any $\varepsilon>0$,
	there exist (see e.g. \cite{Stow,ZZ}) a neighbourhood  $\mathcal{U}\subset \R^2$ of $x \simeq (0,0)$, a neighbourhood $\mathcal{V} \subset \R^2$ of $(0,0)$, a linear isomorphism $P\in \mathrm{SL}(2,\R)$, and a $C^{1,\frac 12}$-diffeomorphism $\Psi\colon \mathcal{U} \to \mathcal{V}$,  such that 
	\begin{equation*}%\label{fseccc conjuugg}
	\Psi \circ \mathcal{F} \circ \Psi^{-1} = D_\lambda, \qquad \|\Psi-P\|_{C^1} \leq \varepsilon, \qquad \|\Psi^{-1}-P^{-1}\|_{C^1} \leq \varepsilon,
	\end{equation*}
	with $D_\lambda:=\mathrm{diag}(\lambda,\lambda^{-1})\colon (\xi,\eta)\mapsto (\lambda\xi,\lambda^{-1} \eta)$, and 
	\begin{align*}%\label{contole non}
	\Psi(z)-\Psi(z')&=P(z-z')+O(\max(|z|^{\frac 12},|z'|^{\frac 12})|z-z'|),\\
	\Psi^{-1}(z)-\Psi^{-1}(z')&=P^{-1}(z-z')+O(\max(|z|^{\frac 12},|z'|^{\frac 12})|z-z'|).
	\end{align*}
	As $n \to +\infty$, the points in the orbit $h_n$ get  closer and closer  to those in $h_\infty$. Since the orbit $h_\infty$ is homoclinic to $\mathcal{O}$, for $k \in \Z$, the first return time $\tau(\mathcal{F}^k (x_\infty^1))$  is exponentially close to $\mathcal{L}^{(0)}=\mathcal{L}(x)$ with respect to $|k|$. The total error $\mathcal{L}^{(1)}$ is obtained by adding up the discrepancies $\delta_k:=\tau(\mathcal{F}^k (x_\infty^1))-\mathcal{L}^{(0)}$ over all the points $(\mathcal{F}^k (x_\infty))_{k \in \Z}$. Thus,  $\mathcal{L}_n-n \mathcal{L}^{(0)}-\mathcal{L}^{(1)}$ is the sum of two terms:
	\begin{enumerate}
		\item the sum $\sum_{k=-\lceil n/2 \rceil}^{n+1-\lceil n/2 \rceil}\tau(\mathcal{F}^k (x_n^1))-\tau(\mathcal{F}^k (x_\infty^1))$;
		\item the tail $\sum_{\substack{k > n+1-\lceil n/2 \rceil\\ \text{or }
		k <-\lceil n/2 \rceil}}\mathcal{L}^{(0)}-\tau(\mathcal{F}^k (x_\infty^1))$ of the series $-\sum_k \delta_k$. 
	\end{enumerate}
In order to evaluate the first of these two terms, we need to estimate the coordinates of the points  $\mathcal{F}^k (x_n^1)$, for $k=-\lceil n/2 \rceil,\dots,n+1-\lceil n/2 \rceil$. Except a finite number $N\geq 0$ of them (uniform with respect to $n \geq 0$), all of these points are in the neighbourhood $\mathcal{U}$ where $\Psi$ is defined. In particular, there exist $y_\infty^+,y_\infty^-\in \mathcal{V}\cap \Psi(h_\infty)$ such that for all $n \gg 1$, there exists a point $y_n^+\in \mathcal{V}\cap \Psi(h_n)$  whose iterates $y_n^+,D_\lambda(y_n^+),\dots,y_n^-:=D_\lambda^{n+1-N}(y_n^+)$ are all in $\mathcal{V}$ and such that $\lim_{n \to+\infty} y_n^*=y_\infty^*$, for $*=+,-$ (see Figure \ref{local comp}).  By applying $\Psi^{-1}$, this allows us to estimate the difference $\tau(\mathcal{F}^k (x_n^1))-\tau(\mathcal{F}^k (x_\infty^1))$, for all $k=-\lceil n/2 \rceil,\dots,n+1-\lceil n/2 \rceil$.
\end{proof}

\begin{figure}[H]
	\begin{center}
		\includegraphics [width=13.8cm]{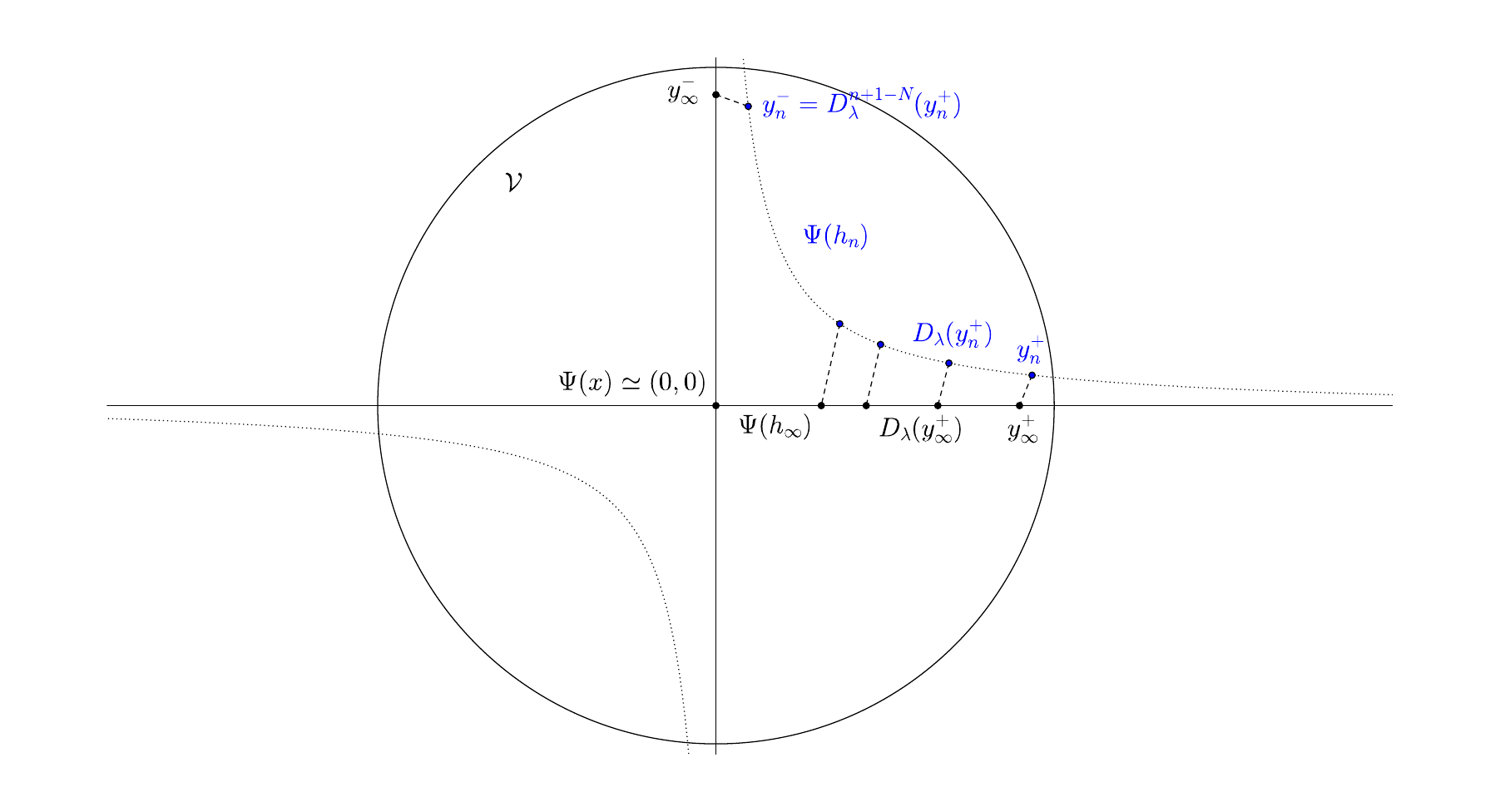}
		\caption{Linearization of the dynamics in a neighbourhood of  $x\simeq (0,0)$.}\label{local comp}
	\end{center}
\end{figure}

The result of Proposition \ref{prop equal le} can be reformulated as follows:
\begin{lemma}
	If the SRB measure is equal to the MME, then for each $n \geq 0$, the Floquet multipliers of $D\Phi_{x_n}^{\mathcal{L}_n}$ are equal to $e^{\pm h_{\mathrm{top}}(\Phi)\mathcal{L}_n}$.\footnote{They are also equal to the eigenvalues of $D\mathcal{F}_{x_n}^{n+2}$, see for instance Lemma 1 on p.111 of \cite{HZ}.}
\end{lemma}

%\begin{proof}
%	On the one hand, the SRB measure is equal to the equilibrium state when the potential is the logarithm of the unstable Jacobian. Besides, we have here that the MME is equal to the equilibrium state when the potential is the roof function $\tau$. By Bowen (), since SRB=MME, we deduce that their potentials are cohomologous, i.e.,   for any $z \in S$, it holds 
%	$$
%	h\tau (z)=\log \|D \mathcal{F}|_{E_\mathcal{F}^u(z)}\|+ \psi(z) - \psi \circ \mathcal{F}(z),
%	$$
%	where $h>0$ is the topological entropy, %$c$ a constant, 
%	and $\psi\colon M \to \R$ is some function. In particular, for any $n \geq 0$, adding up the above identites over the different points in the orbit of $x_n$ yields: 
%	$$
%	h T_n=\log  \|D \mathcal{F}^n|_{E_\mathcal{F}^u(x_n)}\|=\log  \|D \Phi^{T_n}_{x_n}|_{E_\Phi^u(x_n)}\|,
%	$$ 
%	which concludes. 
	%where $\mathcal{L}(x_n)$ denotes the perimeter of the orbit of $x$.
%\end{proof}

\begin{corollary}
		If the SRB measure is equal to the MME, then as $n \to +\infty$, we get the asymptotic expansion 
	\begin{equation*}
	e^{\pm h_{\mathrm{top}}(\Phi)\mathcal{L}_n}=C_\infty^{\pm 1} \lambda^{\pm n}(1+O(\lambda^n)),
	\end{equation*}
	where $C_\infty:=e^{h_{\mathrm{top}}(\Phi)\mathcal{L}^{(1)}}>0$.  
	
	In particular, as $\mathrm{tr}(D\mathcal{F}_{x_n}^{n+2})=\mathrm{tr}(D\Phi_{x_n}^{\mathcal{L}_n})-1=e^{- h_{\mathrm{top}}(\Phi)\mathcal{L}_n}+e^{ h_{\mathrm{top}}(\Phi)\mathcal{L}_n}$, it holds 
	\begin{equation}\label{exp lambda tn}
	2\cosh(\mathcal{L}_n \mathrm{LE}(h_n))=2\cosh(\mathcal{L}_n h_{\mathrm{top}}(\Phi) )=C_\infty^{-1} \lambda^{-n}+O(1).
	\end{equation}
\end{corollary}

On the other hand, we can obtain a general expression for the Lyapunov exponents of the orbits $(h_n)_{n \geq 0}$, which holds without any assumption on the MME with respect to the SRB measure. Let us recall that for  each large integer $n \gg 1$, we denote by  $(\xi_n,\eta_n)$ the coordinates of the point $R(x_n^2)$ on $\Gamma_2$.  By Lemma \ref{choice of R}, and since $N^{n}  (\xi_n,\eta_n)=N^{n} \circ R(x_n^2)=R(x_n^1)=(\eta_n,\xi_n)$, the pair $(\xi_n,\eta_n)$ is defined implicitely by the following system of equations:
\begin{equation*}
\left\{
\begin{array}{rcl}
\xi_n&=&\gamma(\eta_n),\\
\eta_n&=&\Delta(\xi_n\eta_n)^{n} \xi_n=\Delta(\gamma(\eta_n)\eta_n)^n \gamma(\eta_n).
\end{array}
\right.
\end{equation*}
In particular, this implies:
\begin{equation}\label{xi n eta n exp first}
\xi_n=\xi_\infty+O(\lambda^n),\quad\eta_n=\xi_\infty \lambda^n+O(n\lambda^{2n}). 
\end{equation}
Following the computations in \cite[Lemma 4.10]{DKL}, we obtain:
\begin{lemma}\label{lemme tr dfn}
	As $n \to +\infty$, we have the asymptotic expansion 
	\begin{equation}\label{expansion ly exp}
	2\cosh(\mathcal{L}_n \mathrm{LE}(h_n))=\mathrm{tr}(D\mathcal{F}_{x_n}^{n+2})=C_0 \lambda^{-n}+ n C_1 a_1 + O(1),
	\end{equation}
	where $C_0,C_1 \in \R^*$ are  nonzero constants, and $a_1\in \R$ is the first Birkhoff invariant at $\mathcal{O}$ of $\mathcal{F}$. 
\end{lemma}

\begin{proof}
	Note that when the integer $n \gg 1$ gets large, the orbit $h_n$ shadows the homoclinic orbit $h_\infty$ very closely. Thus, as in \cite{DKL}, we can replace the dynamics of $\mathcal{F}$ near  $h_\infty$ with that of the Birkhoff Normal Form $N$ and some gluing map $\mathcal{G}=R \circ \mathcal{F}^2\circ R^{-1}\colon O_\infty^1 \to R(\mathcal{S}_\infty^2)$ defined on an open subset  $\{(0,\xi_\infty)\}\in O_\infty^1 \subset R(\mathcal{S}_\infty^1)$. 
	In particular, as  the trace $\mathrm{tr}(D\mathcal{F}_{x_n}^{n+2})$ is invariant under the change of coordinates $R$, for any large integer $n \gg 1$, we get\footnote{Recall that $
	N^{n} \circ \mathcal{G}(\eta_n,\xi_n)=(\eta_n,\xi_n)$.} 
	$$
	\mathrm{tr}(D\mathcal{F}_{x_n}^{n+2})=\mathrm{tr}( D(N^{n} \circ \mathcal{G})_{(\eta_n,\xi_n)})=\mathrm{tr}( DN^{n}_{(\xi_n,\eta_n)}  D\mathcal{G}_{(\eta_n,\xi_n)}).
	$$ 
	Moreover, 
	we have 
	$$
	DN^n_{(\xi,\eta)}=\begin{bmatrix}
	\Delta(\xi\eta)^n + n \Delta'(\xi\eta)\Delta(\xi \eta)^{n-1}\xi \eta & n  \Delta'(\xi\eta)\Delta(\xi \eta)^{n-1}\xi^2 \\ 
	-n  \Delta'(\xi\eta)\Delta(\xi \eta)^{-n-1}\eta^2 & \Delta(\xi\eta)^{-n} - n \Delta'(\xi\eta)\Delta(\xi \eta)^{-n-1}\xi \eta
	\end{bmatrix}.
	$$
	By \eqref{xi n eta n exp first}, $\xi_n \eta_n=\xi_\infty^2 \lambda^n +O(n\lambda^{2n})$; besides, $\Delta(\xi_n \eta_n)=\lambda + a_1 \xi_\infty^2 \lambda^n +O(n\lambda^{2n})$, and $\Delta'(\xi_n \eta_n)=a_1 + O(\lambda^n)$, hence
	$$
	DN_{(\xi_n,\eta_n)}^n=\begin{bmatrix}
	\lambda^n & 0 \\ 
	0 & \lambda^{-n}
	\end{bmatrix}-\lambda^{-1}\xi_\infty^2 a_1 n  \begin{bmatrix}
	-2 \lambda^{2n}+O(n\lambda^{3n}) & -\lambda^n+O (n\lambda^{2n})\\
	\lambda^n+O (n\lambda^{2n}) & 2 +O(n\lambda^{n})
	\end{bmatrix}.
	$$
	Moreover, for $\gamma_1=\gamma'(0)$, and some  $g_0 \in \R$, it holds: (see formula (4.4) in \cite{DKL})
	$$
	D\mathcal{G}_{(\eta_n,\xi_n)}=D \mathcal{G}_{(0,\xi_\infty)}+O(\lambda^n)=\begin{bmatrix}
	\gamma_1(2-\gamma_1 g_0) & \gamma_1 g_0-1 \\
	1-\gamma_1 g_0 & g_0
	\end{bmatrix}+O(\lambda^n).
	$$
	We deduce that 
	\begin{align*}
	\mathrm{tr}(D\mathcal{F}_{x_n}^{n+2})&=\mathrm{tr}\bigg(\begin{bmatrix}
	0 & 0 \\ 
	0 & \lambda^{-n}- 2\lambda^{-1} \xi_\infty^2 a_1 n
	\end{bmatrix}\begin{bmatrix}
	\gamma_1(2-\gamma_1 g_0) & \gamma_1 g_0-1 \\
	1-\gamma_1 g_0 & g_0
	\end{bmatrix}\bigg)+O(1)\\
	%&- 2n\lambda^{-1}\xi_\infty^2 a_1 \mathrm{tr}\bigg(\begin{bmatrix}
	%O(\lambda^{2n}) & O(\lambda^n) \\ 
	%O(\lambda^{n}) & 1+O(n\lambda^n)
	%\end{bmatrix}\begin{bmatrix}
	%\gamma_1(2-\gamma_1 g_0) & \gamma_1 g_0-1 \\
	%1-\gamma_1 g_0 & g_0
	%\end{bmatrix}\bigg)\\
	&=g_0 \lambda^{-n}-2 \lambda^{-1} \xi_\infty^2 g_0 a_1 n+O(1).
	\end{align*} 
	In particular, the constants $C_0,C_1\in \R$ in  \eqref{expansion ly exp} are given by 
	$$
	C_0:=g_0,\quad C_1:=-2\lambda^{-1} \xi_\infty^2 g_0.
	$$
%	Let us assume by contradiction that $g_0=0$. Then 
Since the gluing map $\mathcal{G}$ is dynamically defined, the vector $(1,w_1)$ tangent to $\mathcal{W}_\infty^1$ at $(0,\xi_\infty)$ is mapped to a vector $(*,1-g_0(\gamma_1-w_1))$ in the stable subspace at $(\xi_\infty,0)$, which is horizontal in our coordinate system. We deduce that $g_0(\gamma_1-w_1)=1$, and by Lemma \ref{transverse},   $w_1\neq \gamma_1\neq\infty$, so that
\begin{equation}\label{g zero}
g_0=(\gamma_1-w_1)^{-1}\neq 0,
\end{equation}
hence $C_0,C_1 \neq 0$ (note that $\xi_\infty\neq 0$). 
\end{proof}

\begin{corollary}\label{coro birhffofof}
	If the SRB measure is equal to the MME, then for any periodic orbit $\mathcal{O}$, the first Birkhoff invariant at  $\mathcal{O}$ of the  Poincar\'e map $\mathcal{F}$  vanishes. 
\end{corollary}

\begin{proof}
	We consider the periodic point $x$ as above and use the notation introduced previously.  By equating the expressions obtained in
	\eqref{exp lambda tn} and \eqref{expansion ly exp}, for $n \gg 1$, we obtain
	$$
	2\cosh(\mathcal{L}_n \mathrm{LE}(h_n))%\mathrm{tr}(D\mathcal{F}_{x_n}^{n+2})
	=C_\infty^{-1} \lambda^{-n} +O(1)=C_0 \lambda^{-n}+ n C_1 a_1 + O(1).
	$$
	Since $C_1 \neq 0$, we deduce that  $a_1=0$, as desired. 
\end{proof}

By Theorem \ref{theorem HK}, we deduce:
\begin{corollary}\label{coro srb mme}
		If the SRB measure is equal to the MME, then the Anosov class $[A_\Phi]$ vanishes. 
\end{corollary}

%\begin{proof}
%	We refer to Hurder-Katok \cite{HuKa}. Indeed, the first Birkhoff invariant of the Poincar\'e map $\mathcal{F}$ introduced above coincides with the  Anosov cocycle. Moreover, to show that the Anosov cocycle vanishes identically, by Livsic's theorem, it is enough to show that it vanishes at periodic orbits. 
%\end{proof}

Gathering the previous observations, we thus obtain:
\begin{corollary}\label{main coro o}
	Let $k \geq 5$ be some integer, and let $\Phi$ be a volume-preserving $C^k$  Anosov flow on some $3$-dimensional compact connected manifold $M$.  
	
	\begin{enumerate}
		\item\label{first item}\noindent  If the  MME of $\Phi$ is equal to the volume measure, then the weak stable/unstable distributions $E_\Phi^{cs}/E_\Phi^{cu}$ are $C^{k-3}$.  
		\item\label{second item bis}\noindent If the MME of $\Phi$ is equal to the volume measure, then $\Phi$  is $C^k$-orbit equivalent to an algebraic flow. %In other words, there exists a $C^\infty$ diffeomorphism which sends the orbits of $\Phi$ to those of the ``diagonal'' flow on some homogeneous space $\Gamma \backslash \widetilde{\mathrm{SL}}(2,\R)$. 
		\item\label{second item}\noindent If $\Phi$ is obtained by suspending an Anosov diffeomorphism $F$ of the $2$-torus $\T^2$, then the MME of $\Phi$ is equal to the volume measure if and only if $F$ is $C^{k-3}$ conjugate to a linear automorphism.
	%	\item\label{third item}\noindent If $(S,g)$ is a closed orientable surface endowed with a Riemannian metric $g$ such that the geodesic flow $\Phi$ on $T^1 S$ is Anosov, then the MME of $\Phi$ is equal to the volume measure if and only if the metric $g$ has constant negative curvature. 
	\end{enumerate} 
	
\end{corollary}

\begin{proof}
	Item \ref{first item}, resp. \ref{second item}, %resp. \ref{third item} 
	follows from  Corollary \ref{coro srb mme} and Theorem \ref{theorem HK} (Theorem 3.4/Corollary 3.5 in \cite{HuKa}), resp. Corollary 3.6  %resp. Corollary 3.7 
	of Hurder-Katok \cite{HuKa}.  Besides, Item \ref{second item bis} follows from Item \ref{first item} and Theorem \ref{theoreme ghys} (Th\'eor\`eme 4.6 in Ghys \cite{G2}). 
\end{proof}

\subsection{From smooth orbit equivalence to smooth flow conjugacy}\label{subs orb conj}

In this subsection, as in Corollary \ref{main coro o}, let  $k \geq 5$, and let $\Phi$ be a  $C^k$  Anosov flow   on some $3$-dimensional manifold $M$ which preserves a smooth volume $\mu$.

The following result explains how we can upgrade the smooth equivalence obtained in Corollary \ref{main coro o} to a smooth flow conjugacy.

\begin{prop}\label{main prop yang}
	If the MME of $\Phi$ is equal to the volume measure $\mu$, then for any $\varepsilon>0$, $\Phi$ is $C^{k-\varepsilon}$ conjugate to an algebraic flow. 
\end{prop}

\begin{proof}
	Since the MME is equal to the volume measure $\mu$, by Corollary \ref{main coro o}, we know that $\mathcal{W}^{cs}_\Phi$ and $\mathcal{W}^{cu}_\Phi$ are $C^{k-3}$ foliations, and the flow $\Phi$ is  orbit equivalent to an algebraic flow $\Psi$ through a $C^k$ map $\mathcal{H}$. %Let $\{\mu^{cu}_x\}_{x \in M}$ be a system of conditional measures of $\mu$ for the weak unstable foliation $\mathcal{W}^{cu}_\Phi$. As $\mu$ coincides with the Bowen-Margulis measure, the disintegrations satisfy
	%\begin{equation}\label{Margulis e}
	%(\Phi^t)_* \mu_x^{cu}=e^{h t} \mu_{\Phi^t(x)}^{cu},\quad \forall\, x \in M,\, \forall\, t \in \R,
	%\end{equation}
	Up to a linear time change, without loss of generality, we can assume that the topological entropies of $\Phi,\Psi$ coincide, i.e. $h_{\mathrm{top}}(\Phi)= h_{\mathrm{top}}(\Psi)=h>0$. Moreover, up to a $C^k$ conjugacy, the flow $\Phi$ can be seen as a reparametrization of the algebraic flow  $\Psi$. %, i.e.,  for some $C^k$   function $\mathcal{T} \colon M \times \R \to \R$ such that $\partial_2 \mathcal{T}>0$,  we have  $\Phi^t(\cdot)=\Psi^{\mathcal{T}(\cdot,t)}(\cdot)$, for any $t \in \R$.
	%where $h:=h_{\mathrm{top}}(\Phi)>0$ is the topological entropy of $\Phi$. 
	
	Let $x \in M$ be  a point  in some periodic orbit $\mathcal{O}$ of $\Phi$,   of period $\mathcal{L}_\Phi(\mathcal{O})=\mathcal{L}_\Phi(x)>0$. The orbit $\mathcal{O}$  is also periodic for $\Psi$,  of period $\mathcal{L}_\Psi(\mathcal{O})=\mathcal{L}_\Psi(x)>0$.  %=\mathcal{T}(x,\mathcal{L}_\Phi(x))>0$.  %On the one hand, by \eqref{Margulis e} and  $\Phi^{\mathcal{L}(x)}(x)=x$,  we have 
	%\begin{equation}\label{Margulis per}
	%\mathcal{F}_* \mu_x^{cu}=e^{h \mathcal{L}(x)} \mu_{\Phi^{\mathcal{L}(x)}(x)}^{cu}=e^{h \mathcal{L}(x)} \mu_{x}^{cu}.
	%\end{equation}
	%On the other hand, as $\mu$ is the volume measure, we also have $\mathcal{F}_* \mu_x^{cu}=J_ x^u(\mathcal{L}(x)) \mu_x^{cu}$, where $J_ x^u(\mathcal{L}(x)) $ is the Jacobian of the map $D \Phi^{\mathcal{L}(x)}|_{E^u_\Phi (x)}$. Together with \eqref{Margulis per}, we conclude that 
	As the MME is equal to the volume measure $\mu$,  Proposition \ref{prop equal le} yields 
	\begin{equation}\label{first eq qeq}
	J_ {\Phi,x}^u(\mathcal{L}_\Phi(x))=e^{h \mathcal{L}_\Phi(x)},
	\end{equation}
	where $J_ {\Phi,x}^u(\mathcal{L}_\Phi(x))>1$ is the unstable Jacobian of $\Phi$ at $x$. 
	
	%The point $\mathcal{H}(x)$ belongs to some periodic orbit $\mathcal{H}(\mathcal{O})$ of $\Psi$,   of period $\mathcal{L}_\Psi(\mathcal{H}(\mathcal{O}))=\mathcal{L}_\Psi(\mathcal{H}(x))>0$.  
	Similarly,  for the algebraic flow $\Psi$, we have 
	\begin{equation}\label{second eq qeq}
	J_ {\Psi,x}^u(\mathcal{L}_\Psi(x))=e^{h \mathcal{L}_\Psi(x)}. 
	\end{equation}
	
	Let us take a smooth  section $\mathcal{S}_x$ transverse to the flows $\Phi,\Psi$ at $x$. As $\Phi$ is a smooth reparametrization of $\Psi$, the two flows induce the same Poincar\'e map on $\mathcal{S}_x$, which we denote by $\mathcal{F}$. When $\mathcal{F}$ is seen as the Poincar\'e map of $\Phi$, resp. $\Psi$,  the eigenvalues of its differential $D\mathcal{F}_x$ (which are independent of the choice of the point $x$ 	and of the transverse section at $x$) are  equal to $(J_ {\Phi,x}^u(\mathcal{L}_\Phi(x)))^{\pm 1}$, resp. $(J_ {\Psi,x}^u(\mathcal{L}_\Psi(x)))^{\pm 1}$, hence $J_ {\Phi,x}^u(\mathcal{L}_\Phi(x))=J_ {\Psi,x}^u(\mathcal{L}_\Psi(x))$. Together with \eqref{first eq qeq} and \eqref{second eq qeq}, we conclude that for any periodic orbit $\mathcal{O}$, the associated periods for $\Phi$ and $\Psi$ are equal, i.e., 
	\begin{equation}\label{equality periods}
	\mathcal{L}_\Phi(\mathcal{O})=\mathcal{L}_\Psi(\mathcal{O}).
	\end{equation}
	
	Based on \eqref{equality periods}, we can produce a continuous flow conjugacy following a classical ``synchronization'' procedure, which we now recall.
	
	Let us   denote by $X_\Phi\cdot \mathcal{H}$ the derivative of $\mathcal{H}$ with respect to the flow vector field $X_\Phi$; then for any $x \in M$,   it holds 
	$$
	X_\Phi \cdot  \mathcal{H}(x)=w_{\mathcal{H}}(x) X_\Psi( \mathcal{H}(x)),
	$$
	for some function $w_{\mathcal{H}} \colon M\to \R$ which measures the ``speed" of $\mathcal{H}$ along the flow direction.  
	By \eqref{equality periods}, the function $w_{\mathcal{H}}-1$   integrates to $0$ over all periodic  orbits. By  Livsic's theorem (see \cite{HaKa}, Subsection 9.2), we deduce that there exists a continuous function $u \colon M \to \R$ differentiable along the direction of the flow $\Phi$ such that  $w_{\mathcal{H}}-1=X_\Phi \cdot u$.  Now, let $\mathcal{H}_0\colon x \mapsto \Psi^{- u(x)} \circ \mathcal{H} (x)$.  Given any $x \in M$, we compute 
	\begin{align*}
		w_{\mathcal{H}_0}(x) X_\Psi(\mathcal{H}_0(x))&=X_\Phi \cdot (\Psi^{- u(x)} \circ \mathcal{H})(x)\\
		&=\big(w_{\mathcal{H}} (x) - X_\Phi \cdot  u (x)\big) X_\Psi(\Psi^{- u(x)} \circ \mathcal{H} (x))=X_\Psi(\mathcal{H}_0 (x)),
	\end{align*}
	i.e., $w_{\mathcal{H}_0}\equiv 1$.
	It follows that the homeomorphism $\mathcal{H}_0$ conjugates the flows $\Phi$ and $\Psi$:
	$$
	\mathcal{H}_0 \circ \Phi^t=\Psi^t \circ \mathcal{H}_0,\quad \forall\, t \in \R. 
	$$
	Besides, the Lyapunov exponents of corresponding periodic orbits of $\Phi$ and $\Psi$ coincide (by Proposition \ref{prop equal le},  they are all equal to $h_{\mathrm{top}}(\Phi)= h_{\mathrm{top}}(\Psi)=h$). We conclude from a rigidity result of de la Llave (Theorem 1.1 in \cite{dlL}, see also \cite{dlLM} where the $C^\infty$ case was considered) that the conjugacy $\mathcal{H}_0$ is in fact $C^{k-\varepsilon}$, for any $\varepsilon>0$.  
\end{proof} 

Let us now conclude the proof of Theorem \ref{thm:main bis}. Given $k \geq 5$ and a  $C^k$ Anosov flow $\Phi$ on some three manifold preserving a smooth volume measure $\mu$,  if the  MME of $\Phi$ is equal to $\mu$, then Proposition \ref{main prop yang} says that $\Phi$ is $C^{k-\varepsilon}$ conjugate to an algebraic flow, for any $\varepsilon>0$.  Conversely, assume that $\Phi$ is a   volume-preserving Anosov flow on some three manifold  that  is  conjugate to an algebraic flow $\Psi$ through a smooth map $\mathcal{H}$. As the  MME of $\Psi$ coincides with the SRB measure, and the smooth conjugacy $\mathcal{H}$ takes the MME, resp. the SRB of $\Phi$ to the MME, resp. the SRB of $\Psi$, we conclude that the MME of $\Phi$ coincides with the volume measure, as desired. \qed

\section{Entropy Flexibility}
\label{sec:flexibility}

Theorem \ref{thm:main bis} has shown that for an Anosov flow  $\Phi$ on some $3$-manifold preserving a smooth volume $\mu$,  $\htop(\Phi) = h_\mu(\Phi)$ if and only if $\Phi$ is an algebraic  flow up to smooth conjugacy. In particular, the topological entropy $\htop(\Phi) $ and the measure-theoretic entropy $h_\mu(\Phi)$ for the volume measure have to be the same for algebraic flows. It is then natural to ask: %One may be wondering a natural question:
what about non-algebraic flows? Here we show that the numbers of the topological entropy $\htop(\Phi) $ and the measure-theoretic entropy $h_\mu(\Phi)$ with respect to  the volume measure for the suspension flows over Anosov diffeomorphisms are quite flexible for non-algebraic flows.

%In analogy to the case of geodesic flows on surfaces,
We normalize the total volume of the suspension space to $1$ (equivalently, we normalize the integral of $r$ with respect to Lebesgue to be $1$), i.e., $\int r \, d\mu = 1$. One may also think of this as finding a canonical linear time change of an arbitrary flow, and is analogous to fixing the volume of a surface when considering geodesic flows. There are three natural restrictions on $\htop(\Phi) $ and  $h_\mu(\Phi)$. The Pesin entropy formula and positivity of Lyapunov exponents imply that $h_\mu(\Phi) > 0$.  Moreover, the \emph{variational principle} implies \[\htop(\Phi) \ge h_\mu(\Phi).\] Finally, the \emph{Abramov formula} gives 
\[
h_\mu(\Phi) = \frac{h_\mu(f)}{\int r \, d\mu}=h_{\mu}(f) \le \htop(f)=:h.
\] 
Here we denote by $h$ the topological entropy of the torus automorphism in the same homotopy class as $f$. Since any two Anosov diffeomorphisms in the same homotopy class are conjugated with each other \cite{Fr}, we have $h=\htop(g)$, for any $g$ in the same homotopy class as $f$. 

In this section, we shall prove Theorem \ref{prop:flexibility} which says that the pair of values of entropy under the three natural restrictions mentioned above can all be achieved. 
Let $A \in \mathrm{SL}(2,\Z)$ be a hyperbolic matrix whose induced torus automorphism has topological entropy $h$. Let $\mu$ be the volume measure on $\T^2$. Then for any $c_{\operatorname{top}}  > c_\mu> 0$ such that $c_\mu \le h$, we shall find a volume-preserving Anosov diffeomorphism $f \colon \T^2 \to \T^2$  homotopic to $A$ and a $C^\infty$ function $r \colon \T^2 \to \R^{+}$  with integral $1$ with respect to the volume measure such that if $\Phi$ is the suspension flow induced by $f$ and $r$, then $\htop(\Phi) = c_{\operatorname{top}}$ and $h_\mu(\Phi) = c_\mu$.

\noindent\begin{minipage}{.5\textwidth}%{2.5in}
	The figure to the right shows the content of Theorem \ref{prop:flexibility}, where the horizontal axis is $h_{\mu}$ and the vertical axis is $h_{\mathrm{top}}$. The dashed area can be achieved by some  suspension flow.
	The corner point represents the unique flow up to $C^\infty$ conjugacy, namely the algebraic flow. The boundaries are not achievable, with the exception of the right boundary. If we relax the regularity to $C^{1+\alpha}$ the bottom boundary is achievable. %Finally, the discussion preceding Proposition \ref{prop:flexibility} shows that the result is optimal.
\end{minipage} \hspace{.75in} \begin{minipage}{.25\textwidth} 
	\[\begin{tikzpicture} \draw[very thick, dashed] (0, 5) -- (0, 0);
	\draw[ultra thick, dashed] (0, 0) -- (1.5, 1.5); \draw[very thick] (1.5, 1.5) -- (1.5, 5);
	\path[fill = blue, opacity=0.2] (0, 5) -- (0, 0) -- (1.5, 1.5) -- (1.5, 5); \node at (1.5, 1.5) {$\bullet$};
	\node at (1.5, 1.5) [right] {$(h, h)$}; \node at (0, 0) [below] {$(0, 0)$};
	\draw[->] (0, 0) -- (3, 0) node [right] {$h_{\mu}$}; \draw[->] (0, 0) -- (0, 5.25) node [above left] {$h_{\mathrm{top}}$};
	\draw (0, 1.5) -- (1.5, 1.5);
	\node at (.5, 1) {I};
	\node at (.75, 3) {II};  
	\end{tikzpicture}\]%\includegraphics[width=.75in]{flexibility.pdf}
\end{minipage}	

To prove Theorem \ref{prop:flexibility}, we shall use the following lemma, which follows immediately from the variational principle and Abramov formula. 

\begin{lemma}[Theorem 2 in \cite{KKW}]
	\label{lem:htop-continuity}
	Suppose that $f \colon M \to M$ is an Anosov diffeomorphism of a smooth manifold, and $\Phi_r$ is the suspension flow of $f$ with $C^{\infty}$ roof function $r$. Then $r \mapsto \htop(\Phi_r)$ is continuous.
\end{lemma}

%Now we are ready to give the proof of Theorem \ref{prop:flexibility}.
\begin{proof}[Proof of Theorem \ref{prop:flexibility}]
	We first realize the region II. By \cite{HJJ} (see also \cite{E}), for any $h \ge c_\mu > 0$, there exists a volume-preserving Anosov map $f \colon\T^2 \to \T^2$ of homotopy type $A$ such that $h_\mu(f) = c_\mu$. Taking the roof function $r \equiv 1$ gives the case where $c_{\operatorname{top}}=h$.  Since we may homotope the roof function to a constant, it suffices to show that $\htop(\Phi_r)$ can be made arbitrarily large while keeping $\int r = 1$, by Lemma \ref{lem:htop-continuity}. A similar approach also appeared in \cite{EK}. Given $\ve, \delta > 0$, let $r\colon \R^2 \to \R^{+}$ be a $C^\infty$ function satisfying $r \ge \delta$, $ \int r \, d\mu = 1$ and $r|_{\T^2 \setminus B_{\ve}(0)} \equiv \delta$.
	%\begin{enumerate}[(a)]
	%\item $r \ge \delta$;
	%\item $\int r \, d\mu = 1$;
	%\item $r|_{\T^2 \setminus B_{\ve}(0)} \equiv \delta$.
	%\end{enumerate}
	It is clear that such functions exist. Let $\sigma \colon \Sigma_0 \to \Sigma_0$ be a Markov shift coding $f$, with coding map $\pi \colon \Sigma_0 \to \T^2$. By choosing $\ve$ sufficiently small, and a sufficient refinement of a Markov parition for $f$, we can find a subshift $\Sigma \subset \Sigma_0$ such that $\pi(\Sigma) \cap B_\ve(0) = \emptyset$ and $\htop(\sigma|_{\Sigma}) > 0$. This can be constructed easily symbolically by increasing the ``memory'' of $\Sigma_0$ and disallowing the blocks containing $0$, but a more geometric construction can be found in \cite[Corollary 4.3]{K2}. Hence the topological entropy of the flow is at least the topological entropy of the flow restricted to the suspension of this subshift. Since $r$ is identically $\delta$ on $\pi(\Sigma)$, $\htop(\Phi_r) \ge \htop(\sigma|_\Sigma)/\delta$. Since $\delta$ can be arbitrarily small, we get the result.
	
	Now we consider the case where $c_\mu < c_{\operatorname{top}}<h$ (i.e., region I). We start by again taking an Anosov diffeomorphism $f \colon \T^2 \to \T^2$ such that $h_\mu(f) = c_\mu$. If we choose the $C^{1+\alpha}$ roof function $x \mapsto \|Df_x|_{E^u_f}\|$ (normalized to have integral one), we obtain a flow such that $\htop(\Phi_r) = h_\mu(f) = c_\mu$. By perturbing $r$ to a $C^\infty$ roof function $C^1$-close to $r$ and the continuity of topological entropy, we can get a $C^\infty$ roof function with topological entropy arbitrarily close to $c_\mu$. Then taking a linear homotopy of the roof function to the constant one function gives all intermediate values.
\end{proof}

Another natural normalization for the roof function $r$ is to normalize its integral with respect to the measure of maximal entropy of the base. When normalizing with respect to the maximal entropy measure, the natural restrictions become $\htop(\Phi) \ge h$ and  $h_\mu(\Phi)\le \htop(\Phi)$. Theorem \ref{prop:flexibilitymme} tells us all of these numbers can be achieved. 

\noindent\begin{minipage}{.5\textwidth}%{2.5in}
	The figure to the right shows the content of Theorem \ref{prop:flexibilitymme}, where the horizontal axis is $h_{\mu}$ and the vertical axis is $h_{\mathrm{top}}$. The dashed area can be achieved by some   suspension flow.
	The corner point represents the unique flow up to $C^\infty$ conjugacy, namely the algebraic flow. The boundaries are not achievable, with the exception of the bottom boundary. If we relax the regularity to $C^{1+\alpha}$ the right boundary is achievable. %Finally, the discussion preceding Proposition \ref{prop:flexibility} shows that the result is optimal.
\end{minipage} \hspace{.3in} \begin{minipage}{.25\textwidth} 
	\[\begin{tikzpicture} \draw[very thick, dashed] (0, 5) -- (0, 1.5);
	\draw[very thick, dashed] (1.5, 1.5) -- (5, 5);%\draw[very thick] (1.5, 1.5) -- (1.5, 5);
	\path[fill = blue, opacity=0.2] (0, 5) -- (0, 1.5) -- (1.5, 1.5) -- (5, 5); \node at (1.5, 1.5) {$\bullet$};
	\node at (1.5, 1.5) [right] {$(h, h)$}; \node at (0, 0) [below] {$(0, 0)$};
	\draw[->] (0, 0) -- (3, 0) node [right] {$h_{\mu}$}; \draw[->] (0, 0) -- (0, 5.25) node [above left] {$h_{\mathrm{top}}$};
	\draw (0, 1.5) -- (1.5, 1.5);
	%\node at (2, 3) {I};
	%\node at (.75, 3) {II}; 
	\end{tikzpicture}\]%\includegraphics[width=.75in]{flexibility.pdf}
\end{minipage}

The proof of Theorem \ref{prop:flexibilitymme} is significantly more complicated. The reason is that we cannot change the value of topological entropy independently of the value of metric entropy with respect to the invariant volume. This was possible in Theorem \ref{prop:flexibility} because the measure we normalized by was one of the measures being considered. Here, the MME for the base may not (and in most cases, {\it will} not) be the MME for the flow, so in reality, we must control three different measures on the base: the MME of the base, the invariant volume, and the measure which induces the MME of the flow. We shall use $\mu$  to refer to the volume measures and $\nu$ to indicate the corresponding maximal entropy measures.  It is not surprising that we need the following lemma in the proof:

\begin{lemma}
	\label{lem:torus-flexibility}
	Fix any hyperbolic matrix $A \in \mathrm{SL}(2,\Z)$, there exist $\Sigma$, a proper subshift of $A$, a Markov measure $\rho$ on $\Sigma$ with positive entropy such that for every $\delta > 0$, $L > 0$,  there is a one-parameter family $\{f_t\}_{t \in [0,1]}$ of $C^\infty$, area-preserving Anosov diffeomorphisms continuous in the $C^1$ topology such that if $\pi_t \colon \Sigma \to \T^2$ is the continuously varying embedding of the subshift, then
	
	\begin{enumerate}[label=(\roman*)]
		\item\label{premier symb} $f_0$ is the hyperbolic toral automorphism induced by $A$;
		\item\label{deux symb} $h_\mu(f_1) < \delta$;
		\item\label{trois symb} $\lambda_{\nu_1}(f_1) > L$;
		\item\label{quatre symb} $\lambda_{(\pi_1)_*\rho}(f_1) < \delta$.
	\end{enumerate}
\end{lemma}

\begin{proof}
	Parts \ref{premier symb}-\ref{trois symb} of the lemma are exactly the content of \cite{E}. So we must find the subshift $\Sigma$ and measure $\rho$. By Franks-Manning \cite{Fr}, there exists a topological conjugacy between $A$ and $f_t$. %Hence there exists $\Sigma_t$, a proper subshift of $f_t$ that is obtained through the conjugacy $h_t$.  
We shall construct an embedding of subshift for $f_1$ and then $\Sigma$  for $f_t$ and moreover for $A$ will be obtained through the conjugacy. If $p$ is a periodic orbit of $A$, notice that there are associated periodic orbits of $f_t$ obtained through the Franks-Manning conjugacy.  Let $\lambda_p(t)$ denote the multiplier of $p_t$ for $f_t$. We claim that there exists $p \in \T^2$ such that $\lambda_p(1) < \delta$. If this were not true, then the Lyapunov exponent of every periodic measure would converge to something at least $\delta$. Since Lebesgue measure is a Gibbs state, it can be arbitrarily well approximated by a periodic measure. This would imply that the exponent of Lebesgue measure is at least $\delta$, contradicting \ref{deux symb}.
	
	Let $p$ be a periodic point such that $\lambda_p(1) < \delta$, and $q$ denote the period of $p$. 
	%Adjusting the norm as necessary, 
We may without loss of generality assume that $\|D(f_1^q)_p \| = e^{\delta'}$, where $\delta'<\delta$. Choose a Markov partition fine enough so that  any cylinder set $C$ of length $q$ containing $p$ has $\|D(f_1^q)_x\| < e^{2\delta'}$ for all $x \in C$. Now, choose any Markov measure which gives the union of these cylinder sets measure $1-\frac{\delta'}{\max\|Df_1\|}$, and any other cylinder set measure 0. Then the Lyapunov exponent is obtained by integrating the unstable Jacobian with respect to this measure, which will be smaller than $2\delta'(1-\frac{\delta'}{\max\|Df_1\|}) + \frac{\delta'}{\max\|Df_1\|} \cdot \max\|Df_1\|$, which is linear in $\delta'$. Revising our choice of $\delta'$ (to, e.g., $\delta/3$) gives the result. 
\end{proof}

\begin{lemma}
	\label{lem:funny-bump}
	Let $\gamma > 0$, $f \colon \T^2 \to \T^2$ be any $C^\infty$ diffeomorphism not smoothly conjugate to its corresponding automorphism, and $\rho$ be an ergodic Markov measure supported on a subshift of finite type containing a fixed point $x_0$. Then there exists a positive $C^\infty$ function $q \colon \T^2 \to \R$ such that:
	
	\begin{itemize}
		\item $\int q \, d\mu = 1$;
		\item $\int q \, d\nu = 1$, where $\nu$ is the measure of maximal entropy for $f$;
		\item $\int q \, d\rho = \gamma$.
	\end{itemize}
\end{lemma}

\begin{proof}
	For convenience, we denote $\omega_1 = \mu$, $\omega_2= \nu$ and $\omega_3 = \rho$. Then the $\omega_i$ are linearly independent, since each $\omega_i$ is ergodic, $\mu$ and $\nu$ are both fully supported but singular since $f$ is not smoothly conjugate to an automorphism, and $\rho$ is supported on a Cantor set. By our assumptions, we may find some $g \colon \T^2 \to \R$ continuous such that $g_i = \int g \, d\omega_i$ are all distinct. Let $\eta< \frac{1}{2} \min_{i\neq j}\{ |g_i - g_j|\}$, and
	\begin{equation*}
	U_{N,i} :=  \left\{x \in \T^2 : \left|\frac{1}{N} \sum_{k=0}^{N-1} g(f^k(x)) - g_i\right| < \eta\right\}.
	\end{equation*}
	
	By choice of $\eta$, the $U_{N,i}$ are pairwise disjoint for every $N$. Furthermore, by the Birkhoff Ergodic theorem, $\omega_i(U_{N,i}) \to 1$ as $N \to \infty$. Given $\alpha' > 0$, we may choose $N$ sufficiently large so that $\mu(U_{N,i}) > 1-\alpha'$ (which, since $U_{N,i} \cap U_{N,j} = \emptyset$, implies that $\omega_i(U_{N,j}) < \alpha'$). Given $\alpha > 0$, we may choose $\alpha'$ sufficiently small (and hence $N$ sufficiently large) so that $\omega_i(U_{N,j}) < \alpha \omega_i(U_{N,i})$.  Note that the $U_{N,i}$ are open for every $N$. We may take nonnegative functions $\beta_i$ supported in $U_{N,i}$, which have integral $1$ with respect to $\omega_i$. By our construction, the $3 \times 3$ matrix $B$ with entries $b_{ij} = \int \beta_i \, d\omega_j$ is very close to the identity. Its diagonal terms are all $1$ and its off-diagonal terms are all less than $\alpha$.
	
	We wish to find coefficients $a_i$ such that if $\beta = a_1\beta_1 + a_2\beta_2 + a_3\beta_3$, $\int \beta \, d\mu= \int \beta \, d\nu = 1-\gamma/2$ and $\int \beta \, d\rho =\gamma/2$. Such numbers are exactly $B^{-1}$ applied to the vector $(1-\gamma/2,1-\gamma/2,\gamma/2)$. Since the identity fixes this vector and $\gamma$ is fixed, any sufficiently small perturbation will also send this vector to a vector with positive entries. Thus, by choosing $\alpha$ small enough, we know that the corresponding solution has positive coefficients. Since the $\beta_i$ are all nonnegative, adding the constant function $\gamma/2$ to $\beta$ yields the desired function.
\end{proof}

\begin{proof}[Proof of Theorem \ref{prop:flexibilitymme}]
	Let $f_t$ and $\rho$ be as in Lemma \ref{lem:torus-flexibility}, and $q$ be chosen as described in Lemma \ref{lem:funny-bump} for $f_1$. Then we define a family of roof functions $r_{s,t}$ by:
	
	\begin{itemize}
		\item $r_{0,t}(x) \equiv 1$; 
		\item $r_{1,t}(x) := \dfrac{\log J^u_{f_t}(x)}{\lambda_{\nu_t(f_t)}}$, where $J^u_{f_t}\colon x \mapsto \|D(f_t)_x|_{E^u_{f_t}}\|$ and $\nu_t$ is the MME of $f_t$;
		\item $r_{1/2,t}(x) := (1-t) + tq(t)$, where we describe $q(t)$ in Lemma \ref{lem:funny-bump} applied to $f_1$ and $\rho$;
		\item $r_{s,t}(x) := (1-2s) + 2s r_{1/2,t}(x)$ when $s \in [0,1/2]$;
		\item $r_{s,t}(x) := (2-2s)r_{1/2,t}(x) + (2s-1)r_{1,t}(x)$ when $s \in [1/2,1]$.
	\end{itemize}
	
	Let $\phi_{s,t}$ denote the $C^{1+\alpha}$ flow over the base transformation $f_t$ with roof function $r_{s,t}$. We will show the following properties, as illustrated in Figure \ref{fig:mme-flexibility}:
	
	\begin{enumerate}[label=(\alph*)]
		\item\label{ite a} $\htop(\phi_{0, t}) = h$, $h_{\mathrm{Leb}}(\phi_{0, 0}) = h$, and $h_{\mathrm{Leb}}(\phi_{0, 1}) < \delta$; 
		\item\label{ite b} $\htop(\phi_{1, 0})=h_{\mathrm{Leb}}(\phi_{1, 0})=h$;
		$\htop(\phi_{1, t}) =h_{\mathrm{Leb}}(\phi_{1, t})$;  $h_{\mathrm{Leb}}(\phi_{1, t})> L$;
		\item\label{ite c} $\htop(\phi_{1/2, 1}) > L$ and $h_{\mathrm{Leb}}(\phi_{s, 1}) < \delta$ for all $s \in [0, 1/2]$; %\quad and
		\item\label{ite d} $\htop(\phi_{s, 1}) > L$ for all $s \in [1/2, 1]$;
		\item\label{ite e} $\htop(\phi_{s,0}) = \htop(\phi_{1-s,0})$ and $h_\mu(\phi_{s,0}) = h_\mu(\phi_{1-s,0})$ for $s \in [0,1/2]$.
	\end{enumerate}

	\begin{figure}
		\label{fig:mme-flexibility}
		\centering
		\begin{tikzpicture}

		\draw [thick, blue] (-5.5,0) edge (-3.5,0);
		\draw [thick, green] (-5.5,0) edge (-5.5,2);
		\draw [thick, red] (-3.5,0) edge (-3.5,2);
		\draw [thick, purple] (-5.5,2) edge (-3.5,2);
		\draw [dashed] (-4.5,2) edge (-4.5,0);
		\draw [->] (0,0) edge (0,6);
		\draw (0,6) node[anchor=east] {$\htop$};
		\draw [->] (0,0) edge (5.5,0);
		\draw (5.5,0) node[anchor=west] {$h_\mu$};	
		\draw [dotted] (.2,-.5) edge (.2,6);
		\draw (.2,-1) node[anchor=south] {$\delta$};
		\draw [dotted] (0,5) edge (6,5);
		\draw (6,5) node[anchor=west] {$L$};

		\draw [thick,red] (1,1) edge (4.5,5);
		\draw [thick,green] (1,1) edge (.2,1);

		\draw [thick, purple] (.2,1) edge (.2,5);
		\draw [thick, purple] (.2,5) edge (4.5,5);
		\draw [thick, purple] plot[smooth, tension=.7] coordinates {(1,1)};
		\draw [thick,blue] plot[smooth, tension=.7] coordinates {(1,1) (1,1.5) (1.5,2) (1,2.5) };
		\draw [dashed] plot[smooth, tension=.7] coordinates {(1,2.5) (1,4) (0.5,4.5) (.2,5)};
		\draw [->](-2.5,1.7321) arc (119.9993:60:2);

		\filldraw [blue] (1,1) circle (1pt) node[anchor=west] {$(h,h)$};
		\filldraw [black] (1,2.5) circle (2pt);
		\filldraw [black] (-4.5,0) circle (2pt);
		\filldraw [black] (-4.5,2) circle (2pt);
		\filldraw [black] (.2,5) circle (2pt);

		\end{tikzpicture}
		\caption{A $2$-parameter family of flows.}
	\end{figure}
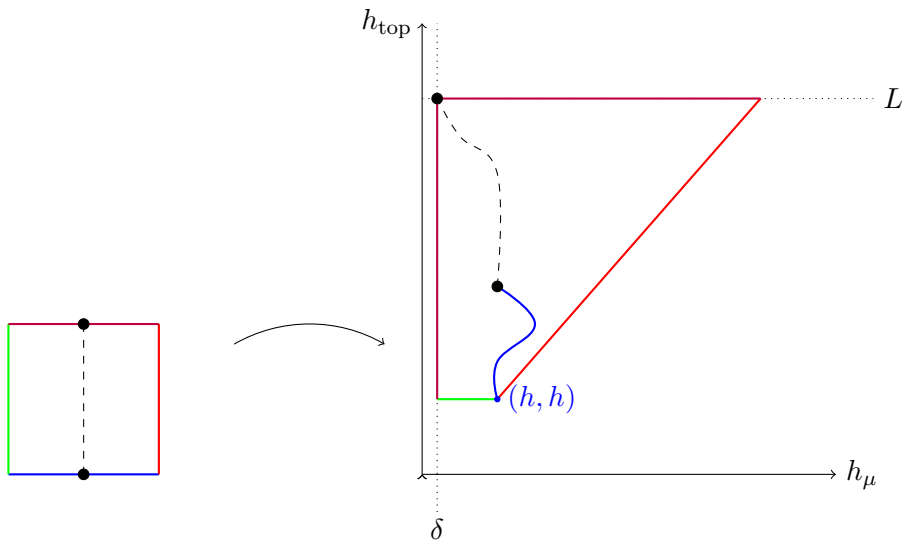

	Let us first see how to apply these properties. We state a topological lemma which follows from standard tools:

\begin{lemma}
\label{lem:topological}
Suppose $\Omega \subset \R^2$ is a simply connected domain, and that $f \colon D^2 \to \R^2$ satisfies $f(\mathbb{S}^1) \cap \Omega = \emptyset$ and $f|_{\mathbb{S}^1}$ is not contractible as a function to $\R^2 \setminus \Omega$. Then $f(D^2) \supset \Omega$.
\end{lemma}

The proof is simple: notice that $f$ as a map from $D^2$ to $\R^2$ is automatically nullhomotopic since $D^2$ is contractible. However, if there existed a point $x_0 \in \Omega$ that was missed by $f$, it is not difficult to see that one may use a linear homotopy so that $f$ restricted to its boundary has the same homotopy type in $\R^2 \setminus \Omega$. Since this is assumed to be noncontractible, we get that every point of the domain $\Omega$ is hit by $f$.

Lemma \ref{lem:topological} together with conditions \ref{ite a}-\ref{ite e} imply that the region in Figure \ref{fig:mme-flexibility} will be filled in. We simply precompose with the function defined through Figure \ref{fig:triangle-to-box}. Notice that the function is not well-defined as described below, but property \ref{ite e} will guarantee well-definedness after composing with $(s,t) \mapsto (h_\mu(\phi_{s,t}),\htop(\phi_{s,t}))$.

	\begin{figure}
		\label{fig:triangle-to-box}
		\centering
		\begin{tikzpicture}

		\draw [thick,blue] (-4.5,0) edge (-4.5,1);
		\draw [dashed] (-4.5,1) edge (-4.5,2);
		\draw [thick,green] (-4.5,0) edge (-5.5,2);
		\draw [thick,red] (-4.5,0) edge (-3.5,2);
		\draw [thick,purple] (-5.5,2) edge (-3.5,2);
		
		\draw [->](-2.5,1.7321) arc (119.9993:60:2);
		
		\draw [thick, blue] (0,0) edge (2,0);
		\draw [thick, green] (0,0) edge (0,2);
		\draw [thick, red] (2,0) edge (2,2);
		\draw [thick, purple] (0,2) edge (2,2);
		\draw [dashed] (1,2) edge (1,0);
		\end{tikzpicture}
		\caption{A domain satisfying Lemma \ref{lem:topological}.}
	\end{figure}
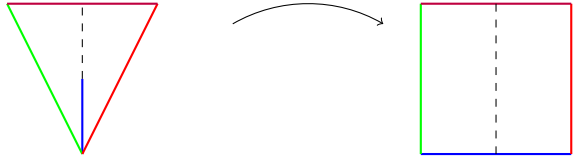

 Since $\delta$ is arbitrarily small and $L$ is arbitrarily large, we get the theorem as claimed. The reader may notice that while the diffeomorphisms obtained are $C^\infty$, the roof functions are only $C^{1+\alpha}$. Perturbing the roof functions in the $C^1$ topology to an approximating $C^\infty$ function fills the interior of the region arbitrarily close to the right-hand boundary, as claimed.

	We now prove claims \ref{ite a}-\ref{ite e}. Note that \ref{ite a} follows immmediately by choice of $f_t$ and since $\htop(f_t) \equiv h$ (they are topologically conjugate to the algebraic model).

	For \ref{ite b}, notice that $\int r_{1,t} \, d\nu_t=1$ for all $t \in [0,1]$.
	Furthermore, since $r_{1,t}$ is proportional to $-\log J^u_{f_t}$, $\mu$ is the measure of maximal entropy for the flow with roof function $r_{1,t}$ and base dynamics $f_t$. Notice that the topological entropy is thus:
	
	\begin{equation*}
	\htop(\phi_{1,t}) = \dfrac{h_\mu(f_t)}{\int r_{1,t} \, d\nu_t} = \lambda_{\nu_t}(f_t).
	\end{equation*}
	
	Hence $\htop(\phi_{1,1}) > L$ by Lemma \ref{lem:torus-flexibility}. Notice also that when $t=0$, $f_0$ is a toral automorphism and $\nu_0=\mu$. Hence we have  $\htop(\phi_{1, 0})=h_{\mu}(\phi_{1, 0})=h$. 
	
	By construction of $q$, $h_\mu(\phi_{s,1}) = h_\mu(f_1) < \delta$ for every $s \in [0,1/2]$. Furthermore, $r_{1/2,1} = q$, so $\htop(\phi_{1/2,1}) \ge h_\rho(f_1)/\int q \, d\rho \ge h_\rho / \gamma$. Thus, by choosing $\gamma$ in Lemma \ref{lem:funny-bump} small enough we get the desired properties.
	
	For \ref{ite d}, notice that $\int r_{s,1} \, d\rho$ is a convex combination of $q$ and $\lambda_{\nu_1}^{-1}\log J^u_{f_1}$. Both have integrals which can be made arbitrarily small, so we may bound the topological entropy of the flows $\phi_{s,1}$ from below using the metric entropy with respect to $\rho$, which is bounded below by $L$.
	
	Finally, for \ref{ite e}, observe that $r_{1,0} = r_{0,0}$ since when $t = 0$, $\log J^u_{f_0} \equiv \lambda_{\nu_0}(f_0)$. Therefore, the base dynamics along this curve stays constant, and the roof function is a homotopy to $r_{1/2,0}$ that returns along the same path.
\end{proof}

\section{Some results towards entropy rigidity for dispersing billiards}\label{section billiards}

\subsection{Preliminaries on dispersing billiards}\label{sec
  prelim}
%
%As in our previous work~\cite{BDSKL}, we will only consider the case
In the following, we consider dispersing billiards of two types:
\begin{enumerate}
	\item\label{open bill} open dispersing billiard tables $\dom:= \R^{2}\setminus \bigcup^{m}_{i = 1}\obs_{i}$, for some integer $m \geq 3$, where $\obs_1,\dots,\obs_m$ are pairwise disjoint closed domains with $C^\infty$ boundary having strictly positive curvature (in particular, they are stricly convex) and satisfying the  \emph{non-eclipse condition}, i.e.,  that the convex hull of any two $\obs_i,\obs_j$, $i\neq j$, is disjoint from the remaining $m-2$ domains;
	\item\label{sinai bill} Sinai billiard tables  $\dom\subset \T^2$  given by
	$\dom = \T^{2}\setminus \bigcup^{m}_{i = 1}\obs_{i}$, for some integer $m \geq 1$, where $\obs_1,\dots,\obs_m$ are pairwise disjoint closed obstacles with $C^\infty$ boundary having strictly positive curvature. 
\end{enumerate}
In either case, we
refer to each of the $\obs_{i}$'s as \emph{obstacle} or
\emph{scatterer}. We let $\ell_i:=|\partial\obs_{i}|$ be the
corresponding lengths,  and set $\T_i:=\R/\ell_i\Z$. We also denote by $|\partial \mathcal{D}|:=\sum_{i=1}^m \ell_i$ the total perimeter of the boundary of $\mathcal{D}$.

For  a fixed integer $m\geq 1$,  the set of all  billiard tables of type \eqref{open bill} or    \eqref{sinai bill} will be denoted by $\billiards^{(1)}(m)$, $\billiards^{(2)}(m)$ respectively.  
Let  $i\in \{1,2\}$ and let 
$\dom \in
\billiards^{(i)}(m)$. We denote by   $\obs_1,\dots,\obs_m$ the collision, and we introduce the collision space 
\begin{align*}
  \mathcal{M} &= \bigcup_i \mathcal{M}_i, &% \\
  \mathcal{M}_i &=\{(q,v),\ q \in \partial\obs_{i},\ v\in \R^2,\ \|v\|=1,\ \langle v,n\rangle\geq 0\},
\end{align*}
where $n$ is the unit normal vector to $\partial\obs_{i}$ pointing
inside $\mathcal{D}$. For each $x=(q,v) \in \mathcal{M}$, $q$ is
associated with the arclength parameter $s \in [0,\ell_i]$ for some
$i\in \{1,\cdots,m\}$,  and we let
$\varphi\in [-\frac{\pi}{2},\frac{\pi}{2}]$ be the oriented angle
between $n$ and $v$. In other words, each $\mathcal{M}_i$ can be seen
as a cylinder $\T_i \times [-\frac{\pi}{2},\frac{\pi}{2}]$ endowed
with coordinates $x=(s,\varphi)$.  

Set $\Omega :=\{(q,v) \in \mathcal{D} \times \mathbb{S}^1\}$. The billiard flow $\Phi=(\Phi^t)_{t \in \R}$ on $\Omega$  is the motion
of a point particle traveling in $\mathcal{D}$ at unit speed and undergoing elastic  reflections at
the boundary of the scatterers (by definition, at a  grazing collision, the
reflection does not change the direction of the particle).  
A key feature is that, although the billiard flow is continuous if one identifies
outgoing and incoming angles, the tangential collisions give rise to singularities in the derivative \cite{CM}.  Let 
$$
\mathcal{F}=\mathcal{F}(\mathcal{D}) \colon \mathcal{M} \to \mathcal{M},\quad x \mapsto \Phi^{\tau(x)+0}(x)
$$
be the associated billiard \textit{map}, where
$\tau\colon\mathcal{M} \to \R_+\cup \{+\infty\}$ is the first return time.

A periodic orbit is called \emph{regular} if it does not experience grazing collisions. 
For any  regular periodic orbit $\sigma=(x_1,\cdots, x_p)$ of period $p \geq 2$, we have
$D_{x_j}\mathcal{F}^p\in \mathrm{SL}(2,\R)$, for
$j \in \{1,\cdots,p\}$.\footnote{Recall that for
	$x=(s,\varphi)\in \mathcal{M}$ and
	$x'=(s',\varphi'):=\mathcal{F}(s,\varphi)$, we have
	$\det D_{x}\mathcal{F}=\frac{\cos \varphi}{\cos \varphi'}$. Thus,
	for any periodic orbit $\mathcal{O}=(x_1,x_2,\cdots ,x_p)$ of period $p \geq 2$,
	we have $\det D_{x_j}\mathcal{F}^p=1$, for $j \in \{1,\cdots,p\}$.}
Due to the strict convexity of the obstacles, $D_{x_j}\mathcal{F}^p$
is hyperbolic, and we let 
$0<\lambda<1<\lambda^{-1}$ be its
eigenvalues. The \textit{Lyapunov exponent} of this orbit for the map $\mathcal{F}$ is defined as
\begin{equation*}%\label{eq:le-definition}
	\mathrm{LE}(\sigma):=-\frac 1p \log
	\lambda >0.
\end{equation*} 

\subsection{Open dispersing billiards}\label{seubc dis open}

We consider a table $\dom=\R^2 \setminus \cup_{i=1}^m \obs_i \in
\billiards^{(1)}(m)$, for some integer $m\geq 3$.  The nonwandering set $\Omega$ of the  billiard map $\mathcal{F}$   is homeomorphic to a Cantor set, and the restriction of $\mathcal{F}$  to   $\Omega$  is conjugated to a subshift of finite type associated with
the transition matrix $(1-\delta_{ij})_{1 \leq i,j \leq m}$, where $\delta_{ij}=1$ if $i=j$, and $\delta_{ij}=0$ otherwise. 
In other words, any word $(\varsigma_j)_{j}\in \{1,\cdots,m\}^\Z$ such that
$\varsigma_{j+1}\neq \varsigma_j$ for all $j \in \Z$ can be realized
by a unique orbit, where the different symbols represent the obstacles on which the successive collisions happen. Such a word is called \textit{admissible}.  
In particular, any periodic orbit of period $p$ (observe that
necessarily $p\ge 2$) can be labeled by a periodic admissible word $\sigma^\infty:=\dots\sigma\sigma\sigma\dots$, for some finite word 
$\sigma=(\sigma_1\sigma_2\dots\sigma_p)\in \{1,\cdots,m\}^p$.  For more regarding this class of open dispersing billiards, we refer the reader to \cite{BDKL}. 

The symbolic coding gives a convenient way to identify homoclinic orbits. Indeed, let us consider any periodic point 
\begin{equation*}
	x\, \longleftrightarrow\,
	 \dots \sigma_p| \sigma_1\dots \sigma_p|\underset{\substack{\uparrow}}{\sigma_1}\sigma_2\dots \sigma_p | \sigma_1\dots \sigma_p|\sigma_1\dots
\end{equation*} 
for some finite word 
$\sigma=(\sigma_1\sigma_2\dots\sigma_p)\in \{1,\cdots,m\}^p$, $p \geq 2$.

 %If $p=2$, take $\tau \neq \sigma_1,\sigma_2$; then the point 
%\begin{equation*}
%	x_\infty\, \longleftrightarrow\,
%	\dots | \sigma_1  \sigma_2| \tau \sigma_2|\underset{\substack{\uparrow}}{\sigma_1}\sigma_2| \sigma_1  \sigma_2|\sigma_1\sigma_2|\dots
%\end{equation*} 
%is homoclinic to $x$. 

%Otherwise, $p \geq 3$.  In this case, we can

Let us  take a word $\tau=(\tau_1\tau_2\dots\tau_p)\in \{1,\cdots,m\}^p$ such that the following word is admissible and defines a point that is  homoclinic to $x$: 
\begin{equation*}
	\dots \sigma_p| \sigma_1\dots \sigma_p| \tau_1\dots \tau_p|\underset{\substack{\uparrow}}{\sigma_1}\sigma_2\dots \sigma_p | \sigma_1\dots \sigma_p|\sigma_1\dots
\end{equation*} 

In particular, the dynamics  of the billiard flow $\Phi$ is of type Axiom A,  and the spectral decomposition is reduced to one basic set. Moreover,  the nonwandering set is far from the singularities; in particular, every periodic orbit is regular. 
By Proposition \ref{potential MME-SRB}, we know that the restriction of $\Phi$ to the nonwandering set has   a unique SRB measure and a unique MME. Moreover, by Proposition \ref{prop equal le}, equality of the SRB measure and of the MME forces the Lyapunov exponents of periodic points to be equal. Arguing as in Subsection \ref{subsec:horseshoe-homoclinic}, thanks to   Corollary \ref{coro birhffofof}, we conclude:
\begin{theorem}
If  the MME is equal to the SRB measure, then the first Birkhoff invariant of each periodic orbit  vanishes, i.e., $a_1(\sigma)=0$, for any periodic orbit $\sigma$. 
\end{theorem}

\subsection{Sinai billiards}\label{sinnnnnai}

We consider a Sinai billiard table $\dom=\T^2 \setminus \cup_{i=1}^m \obs_i \in
\billiards^{(2)}(m)$, for some integer $m\geq 1$. We assume the boundary of each scatterer $\obs_1,\dots,\obs_m$ is of class $C^\infty$, and that $\mathcal{D}$ has \emph{finite horizon}. 
As previously, we denote by $\mathcal{F}=\mathcal{F}(\mathcal{D})\colon (s,\varphi)\mapsto (s',\varphi')$  the billiard map. It is a local $C^\infty$ diffeomorphism.   Let us also recall that $\mathcal{F}$ preserves a smooth invariant  SRB probability measure $\mu= \frac{1}{2 |\partial \mathcal{D}|}\cos \varphi\,  ds d\varphi$ with respect to which the  dynamics is ergodic, K-mixing, and Bernoulli. 

Due to the existence of grazing collisions, the billiard map $\mathcal{F}$ has singularities, yet in \cite{BD}, Baladi and Demers are able to define a suitable notion of topological entropy $h_*$ for $\mathcal{F}$. They need some quantitative control on the recurrence to the set of singularities, which we recall now. Let $\varphi_0\in \R$ be a number  close to $\pi/2$, and let $n_0 \in \N$ be some integer; a collision is called $\varphi_0$\emph{-grazing} if the absolute value of its angle with the normal is larger than $\varphi_0$. Let $s_0=s_0(\varphi_0,n_0)\in (0,1]$ be the smallest number such that\footnote{Indeed, thanks to the finite horizon assumption, we can choose $\varphi_0,n_0$ such that $s_0<1$.}
$$
\text{any orbit of length }n_0\text{ has at most } s_0 n_0 \text{ collisions which are }\varphi_0\text{ grazing.}
$$
The condition that the authors require  in \cite{BD} is that for a certain choice of $\varphi_0,n_0$, it holds 
\begin{equation}\label{cond h etoile}
h_*> s_0 \log 2.  
\end{equation}

In the following, we assume that \eqref{cond h etoile} holds. Let us recall some of the main results of the work    \cite{BD}:
\begin{theorem}[Theorem 2.4, Proposition 7.13 in \cite{BD}]\label{theorem baladi demers}
The billiard map $\mathcal{F}$ admits a unique invariant Borel probability measure $\mu_*$ of maximal entropy, i.e., $h_{\mu_*}(\mathcal{F})=h_*$. 

Moreover, if $\mu_*$ is equal to the SRB measure, i.e., $\mu_*=\mu$, then all  the regular periodic orbits have the same Lyapunov exponent, i.e., 
\begin{equation}\label{point lyap}
\mathrm{LE}(\sigma)=h_*,\quad \text{for any regular periodic orbit }\sigma. 
\end{equation}
\end{theorem}
Let us note that \eqref{point lyap} is the exact analogue of the result obtained in Proposition \ref{prop equal le} in the case of an Axiom A flow restricted to some basic set.

Let $x$ be a point in a regular periodic orbit for the billiard map $\mathcal{F}$; replacing $\mathcal{F}$ with some iterate, without loss of generality, we may assume that $x$ is a saddle fixed point, i.e., $\mathcal{F}(x)=x$, and we denote by  $0<\lambda < 1 <\lambda^{-1}$  the  eigenvalues of  $D\mathcal{F}_x$. Let us assume that there exists a homoclinic point $x_\infty \in \mathcal{W}_\mathcal{F}^s(x)\cap \mathcal{W}_\mathcal{F}^u(x)$.  As in Subsection \ref{subsec:horseshoe-homoclinic}, we let $\mathcal{S}_0\subset \mathcal{S}$  be a small neighbourhood of the periodic point $x$ encoded by the symbol $0$, and let $\mathcal{S}_1\subset  \mathcal{S}$ be a small neighbourhood of the homoclinic point $x_\infty$ encoded by the symbol $1$, so that the symbolic coding associated to $x_\infty$ is
\begin{equation*}
x_\infty\, \longleftrightarrow\,
\dots 000\underset{\substack{\uparrow}}{1}000\dots
\end{equation*} 
We select the sequence $(h_n)_{n \geq 0}$ of periodic orbits in the horseshoe associated to the homoclinic intersection, whose  symbolic coding is given by
\begin{equation}\label{coding h n}
h_n\, \longleftrightarrow\, \dots\vert \underbrace{0\dots 0}_{n+1}1\vert \underbrace{0\dots 0}_{n+1}1\vert \underbrace{0\dots 0}_{n+1}1\vert \dots%\underbrace{w_0\dots w_0}_{n}).
\end{equation}
For each integer $n\geq 0$, $h_n$ has period $n+2$. 
We denote by $x_n^1$, resp. $x_n^2$, the coordinates of  its periodic approximation in $h_n$, encoded by the symbolic sequence 
\begin{equation*}
x_n^1\, \longleftrightarrow\,
\dots \vert 0\dots 0 \underset{\substack{\uparrow}}{0}1\vert 0\dots 0 1\vert \dots ,\quad  \text{resp. }x_n^2\, \longleftrightarrow\,
\dots \vert 0\dots 0 1\vert \underset{\substack{\uparrow}}{0}0\dots 0 1\vert \dots
\end{equation*}    
As recalled in Subsection \ref{subs anosov class}, if $\mathcal{S}_0$ is taken sufficiently small, there exists a $C^\infty$ volume-preserving change of coordinates
$
R_0 \colon
\mathcal{S}_0  \to \R^2$ which conjugates $\mathcal{F}$ to its  Birkhoff Normal Form 
$N=R \circ \mathcal{F}\circ R^{-1}$:
\begin{align*}
N=N_\Delta\colon (\xi,\eta)\mapsto(\Delta(\xi\eta)\cdot\xi,\Delta(\xi\eta)^{-1}\cdot\eta),
\end{align*}
for some  function
$
\Delta\colon z \mapsto a_0+a_1 z+a_2 z^2+\dots$, with $a_0=\lambda\neq 0
$.    
Moreover, by Lemma \ref{choice of R}, the conjugacy $R$ can be chosen in such a way that for  all $n \geq 0$, 
\begin{equation}\label{change of co}
R(x_{n}^1)=(\eta_n,\xi_n)\in \Gamma_1,\qquad%\in \{\xi=\eta\},
R(x_{n}^2)=(\xi_n,\eta_n)\in \Gamma_2,%\in \{\xi=\eta\},\quad \forall \, n \geq 1.
\end{equation}
where  $\Gamma_1=\{(\eta,\gamma(\eta)):|\eta|\ll 1\}$, $\Gamma_2=\{(\gamma(\eta),\eta):|\eta|\ll 1\}$,
are two smooth arcs which  are  mirror images of each other   under the reflection with respect to the first bissectrix $\{\xi=\eta\}$. We denote by    $\gamma_{0}+\gamma_1 \eta+\gamma_2 \eta^2+\dots$ the Taylor expansion of $\gamma$ at $0$.    When $n \to +\infty$, we have $(\eta_n,\xi_n) \to (0,\xi_\infty)$, with $\xi_\infty=\gamma_0 \neq 0$. 

Let $\mathcal{G}=R \circ \mathcal{F}^2 \circ R^{-1}$ be the gluing map from a neighbourhood of $(0,\xi_\infty)$ to a neighbourhood of $(\xi_\infty,0)$. 
For $|\eta|$ sufficiently small, according to formula (4.4) in \cite{DKL}, it holds 
\begin{align}\label{expre differentielllle}
D_{(\eta,\gamma(\eta))} \mathcal{G}
= \begin{bmatrix}
\gamma'(\eta)(2-\gamma'(\eta)g(\eta)) & \gamma'(\eta)g(\eta)-1\\
1-\gamma'(\eta)g(\eta) & g(\eta)
\end{bmatrix},
\end{align}
for some $C^\infty$ function $g\colon \R \to \R$ whose Taylor expansion at $0$ is denoted by  $g_0+g_{1} \eta+g_2 \eta^2+\dots$   

For each integer $n \geq 0$, we set %$\zeta_n:=\xi_n\eta_n$, and  
$$
\zeta_n:=\xi_n\eta_n,\qquad \Delta_n:=\Delta(\zeta_n),\qquad \Delta_n':=\Delta'(\zeta_n)\zeta_n. %=\sum_{k=1}^{+\infty} k a_k
%\zeta_n^k.
$$
As an immediate consequence of \eqref{coding h n} and  \eqref{change of co}, we have  
    \begin{equation}\label{dela n n}
      \eta_n=\Delta_n^{n} \xi_n=\Delta_n^n \gamma(\eta_n).
    \end{equation}
    
    For any integer $k \geq 0$, we introduce scaled coefficients
    \begin{equation}\label{scaled coefff}%{scaled eta n}
    \ba_k:=\lambda^{-1} a_k \xi_\infty^{2k},\quad
    \bg_k:=\gamma_k \xi_\infty^{k-1},\quad \text{and}\quad
    \bgg_k:=g_k \xi_\infty^k,
    \end{equation}
    with $a_0=\lambda$ and $\gamma_0=\xi_\infty$.  Note that
    $\ba_0=\bg_0=1$ and $\bgg_0=g_0\neq 0$ (see \eqref{g zero}). 

    In the following, for any integer $n \gg 1$, we also let
    \begin{equation}\label{scaled eta n}
    \bar\eta_n:=(\xi_\infty \lambda^n)^{-1}\eta_n.  %\quad %\zeta_n:=\xi_n\eta_n,%=\eta_n(\xi_\infty+\gamma(\eta_n)),
    %\quad \text{and}\quad \bz_n:=(\xi_\infty^2\lambda^{n})^{-1}\zeta_n.  %=(\xi_\infty^2\lambda^{n})^{-1} \xi_n \eta_n.
    \end{equation}
Note that \eqref{dela n n} can be rewritten as
\begin{equation}\label{new claim 5.6}
\bar{\eta}_n=\bar{\Delta}(\lambda^n \bar{\eta}_n \bg(\lambda^n \bar{\eta}_n))^n \bg(\lambda^n\bar{\eta}_n),
\end{equation}
where the Taylor expansions of $\bar{\Delta}$ and $\bar{\gamma}$ at $0$ are respectively given by
$$
\bar{\Delta}(\eta)=1+ \ba_1 \eta^1+\ba_2 \eta^2+\dots\quad\text{and} \quad \bar{\gamma}(\eta)=  1+\bg_1 \eta+ \bg_2 \eta^2+\dots 
$$

Following \cite{DKL}, the Lyapunov exponents of the periodic orbits $(h_n)_{n \geq 0}$ can be expressed in terms of the new coordinates; the triangular structure of the following expansion follows from \eqref{new claim 5.6}. 

\begin{lemma}[Lemma 4.8 and Lemma 4.20 in \cite{DKL}]\label{lemme exposant de lya}
As $n \to +\infty$, the asymptotics of the Lyapunov exponent of the periodic orbit $h_n$ is 
  \begin{equation}\label{dev le}
    2 \lambda^n \cosh((n+2)\mathrm{LE}(h_n))=\mathrm{I}_n+\lambda^n\mathrm{II}_n+\lambda^{2n}\mathrm{III}_n,
  \end{equation}
  where
  \begin{align*}
    \mathrm{I}_n&:=\lambda^{n}\Delta_n^{-n}\big(1-n\Delta'_n\Delta_n^{-1}\big)g(\eta_n),\\
    \mathrm{II}_n&:=2n\Delta'_n\Delta_n^{-1}\big(1-\gamma'(\eta_n)g(\eta_n)\big),\\
    \mathrm{III}_n&:=\lambda^{-n}\Delta_n^{n}\big(1+n\Delta'_n\Delta_n^{-1}\big)\gamma'(\eta_n) \big(2-\gamma'(\eta_n)g(\eta_n)\big).
  \end{align*}
  Moreover, there exists a sequence of real numbers
 $
  (L_{q,p})_{\substack{p=0,\cdots,+\infty\\ q=0,\cdots,p}}
  $ such that   
  \begin{equation}\label{eq:exp-lyap-series}
  2\lambda^{n} \cosh((n+2)\mathrm{LE}(h_n)) %\\
  = \sum_{p=0}^{+\infty}\sum_{q=0}^{p}L_{q,p} n^{q}\lambda^{np}.
  \end{equation}
\end{lemma}
 
 Let us now assume that the measure of maximal entropy $\mu_*$ is equal to the SRB measure $\mu$. 
In particular, by Theorem \ref{theorem baladi demers}, for any integer $n \geq 0$, we have
$$
\mathrm{LE}(h_n)=h_*. 
$$
By \eqref{eq:exp-lyap-series}, we deduce that 
	\begin{equation}\label{eq l q p}
	L_{q,p}=0,\quad \forall\, (q,p)\neq(0,0),(0,2),
	\end{equation}
	while  $L_{0,0}=\lambda^{-2}$ and $L_{0,2}=\lambda^2$.  
	
	As a consequence, we are able to conclude the proof of Theorem \ref{main theo billiards}:\footnote{A similar result also holds in the case of open dispersing billiards.}

\begin{corollary}\label{coro orbite per deux}
	If $\mu_*=\mu$, then the Birkhoff Normal Form $N$ is  linear, i.e.,
	$$
	a_k=0,\quad  \forall\, k \geq 1.
	$$
\end{corollary}

\begin{proof}
	Assume by contradiction that $N$ is not linear, and let $k_0\geq 1$ be smallest such that $a_{k_0}\neq 0$. By \eqref{new claim 5.6}, for $n \gg 1$, we have 
	\begin{align*}
	\eta_n&=\xi_\infty \lambda^n+O(n\lambda^{2n}),\\
	\zeta_n=\xi_n\eta_n&=\xi_\infty^2 \lambda^n +O(n\lambda^{2n}),
	\end{align*}
	and as $z\Delta'(z)=k_0 a_{k_0} z^{k_0}+O(z^{k_0+1})$, we also get 
	\begin{equation}\label{delta n prime}
	%\Delta_n^n&=\lambda^n +O(n\lambda^{2n}),\nonumber \\ 
	\Delta_n'=\zeta_n\Delta'(\zeta_n)=k_0 a_{k_0}\xi_\infty^{2k_0}\lambda^{n k_0}+O(n\lambda^{n(k_0+1)}).
	\end{equation}
	For any series of the form 
	\begin{equation}\label{eqn:star}
	\sum_{k,l \geq 0}\alpha_{k,l}n^k \lambda^{nl}, \text{ for some coefficients } (\alpha_{k,l})_{k,l \geq 0},\tag{T}
	\end{equation} 
	we let 
	$$
	\mathbb{P}_1\Big(\sum_{k,l \geq 0}\alpha_{k,l}n^k \lambda^{nl}\Big):=\sum_{l \geq 0}\alpha_{1,l}n \lambda^{nl}.
	$$
	It follows from  \eqref{new claim 5.6} that $\eta_n$ has an expansion of the form (T) (see \cite{DKL} for more details), and 
	$$
	\mathbb{P}_1(\eta_n)=\xi_\infty \bar{a}_{k_0}n \lambda^{n(k_0+1)}+\text{H.O.T.}
	$$
	Similarly, as $\Delta_n^{-n}=(\Delta(\zeta_n))^{-n}$,  we obtain 
	$$
	\mathbb{P}_1(\Delta_n^{-n})=- \bar{a}_{k_0}n \lambda^{n(k_0-1)}+\text{H.O.T.}
	$$
	Besides, by \eqref{delta n prime}, we have 
	$$
	\mathbb{P}_1(n\Delta_n'  \Delta_n^{-1})=k_0 \bar{a}_{k_0}n \lambda^{nk_0}+\text{H.O.T.}
	$$
	If we look at \eqref{dev le}, as $g_0\neq 0$ (by \eqref{g zero}), and  because of the different weights of $\mathrm{I}_n$, $\mathrm{II}_n$ and $\mathrm{III}_n$ in \eqref{dev le}, %the maximal weight for which some terms  contain some non-constant polynomial in $n$ is  for $n  \lambda^{nk_0}$; moreover, there is a unique  term in the expansion of $2 \cosh(2(n+1)\mathrm{LE}(h_n))$  associated to this weight;  it  is in the term $\lambda^{-n}\mathrm{I}_n$, and it is equal to 
	we obtain 
	\begin{align*}
	&\mathbb{P}_1\Big(2\lambda^n \cosh((n+2)\mathrm{LE}(h_n))\Big)=\mathbb{P}_1(\mathrm{I}_n)+\text{H.O.T.}\\
	&=\lambda^n \Big(\mathbb{P}_1(\Delta_n^{-n}) \cdot 1 \cdot g_0 -\lambda^{-n} \cdot\mathbb{P}_1( n\Delta_n'  \Delta_n^{-1}) \cdot g_0+\lambda^{-n} \cdot 1 \cdot \mathbb{P}_1( g_1\eta_n)\Big)+\text{H.O.T.}\\
	&=\lambda^n\Big(- g_0 \bar{a}_{k_0}n \lambda^{n(k_0-1)}-g_0 k_0 \bar{a}_{k_0}n \lambda^{n(k_0-1)}\Big)+\text{H.O.T.}\\
	&=-g_0 (k_0+1) \bar{a}_{k_0}n \lambda^{n k_0}+\text{H.O.T.}
	\end{align*}
	As a result, we deduce that 
	$$
	L_{1,k_0}=-g_0 (k_0+1) \bar{a}_{k_0} \neq 0,
	$$
	which contradicts \eqref{eq l q p}. 
\end{proof}

\begin{remark}
	Note that when the conclusion of Corollary \ref{coro orbite per deux} is true, some of the above expressions can be simplified. Equation \eqref{dela n n} becomes 
	$$
	\eta_n=  \lambda^n \gamma(\eta_n), 
	$$
	so that $\eta_n$ can be expressed only in terms of the jet of $\gamma$ at $0$. Besides, \eqref{dev le} becomes 
	$$
	\lambda^{-n-2} + \lambda^{n+2}=\lambda^{-n} g(\eta_n) + \lambda^{n} \gamma'(\eta_n) \big(2-\gamma'(\eta_n)g(\eta_n)\big).
	$$
	In particular, the Taylor expansions of $\gamma$ and $g$ at $0$ are strongly correlated. 
\end{remark}

	We note the following is the natural question in view of these results, and the fact that the Birkhoff Normal Form is determined locally for periodic orbits:
	
	\begin{question}
		Do there exist curves $\gamma,\delta\colon [-\ve,\ve] \to \R$ such that $\gamma(0) = 0$, $\delta(0) = 1$, $\gamma'(0) = \delta'(0) = 0$ and $\gamma''(0) < 0 < \delta''(0)$ such that the $2$-periodic orbit of the local billiard between the graphs of $\gamma$ and $\delta$ has linear Birkhoff Normal Form?
	\end{question}
	
	A solution to this question would give a candidate for the local behavior of a ``homogeneous'' hyperbolic billiard. A negative answer would suggest that no such billiard exists. 
	
	\begin{remark}
		Given a periodic orbit with Birkhoff Normal Form $N$ and Birkhoff invariants $(a_k)_{k \geq 0}$, $|a_0| =|\lambda| \in (0,1)$,   for $(\xi,\eta)$ close to $(0,0)$, we have the expansion
		$$
		DN_{(\xi,\eta)}=\begin{bmatrix}
		\lambda & 0 \\ 
		0 & \lambda^{-1}
		\end{bmatrix}-\lambda^{-1} a_1 \begin{bmatrix}
		-2 \lambda \xi\eta & -\lambda \xi^2\\
		\lambda^{-1} \eta^2 & 2 \lambda^{-1} \xi \eta
		\end{bmatrix}+ \mathrm{H.O.T.},
		$$
		where for $\xi \eta\neq 0$, the determinant $-3 \xi^2 \eta^2$ of the second matrix is non-zero.  In particular, $a_1$ vanishes if and only if the quadratic part in the expansion of $DN$ is degenerate. 
		
		In the case of a $2$-periodic orbit, one possibility to show that $a_1 \neq 0$ would thus be to compute the quadratic part in the expansion of $D\mathcal{F}^2$ and show that it is non-degenerate (possibly due to the strict convexity of the obstacles).  
	\end{remark}

	\appendix 
	\section{Bowen-Margulis measure}\label{appendix bowen}

		In this appendix, we  assume that $\Phi$ is a \emph{topologically mixing} smooth Anosov flow  on some compact $3$-manifold $M$. Let us  recall  that by a result of Plante \cite{Pl},   $\Phi$ is topologically mixing if and only if $E_\Phi^s$ and $E_\Phi^u$ are not jointly integrable. As the measure of maximal entropy is unique,    it is given by the construction introduced by Margulis, which we now recall. It is first done by constructing a family of measures $\nu^{cu}$ and $\nu^s$ defined on leaves of the unstable foliation $\mathcal{W}_\Phi^{u}$ and of the stable foliation $\mathcal{W}_\Phi^s$, respectively, such that:
		\begin{equation}\label{eq:margulis-cocycle}
		(\Phi^t)_*\nu^{u} = e^{ht}\nu^{u},\qquad (\Phi^t)_*\nu^s = e^{-ht}\nu^s, 
		\end{equation}
		where $h:=h_{\mathrm{top}(\Phi)}>0$ is the topological entropy of $\Phi$. Moreover, $\nu^{u}$ is invariant by  holonomies along the leaves of $\mathcal{W}^s_\Phi$, while $\nu^s$ is invariant by  holonomies along the leaves of $\mathcal{W}^{u}_\Phi$.  Notice that \eqref{eq:margulis-cocycle} allows $\nu^u$ and $\nu^s$ to be easily extended to measures $\nu^{cu}$ and $\nu^{cs}$ on leaves of the weak unstable foliation $\mathcal{W}_\Phi^{cu}$ and of the weak stable foliation $\mathcal{W}_\Phi^{cs}$, respectively. 
		
		%Notice that \eqref{eq:margulis-cocycle} allows $\mu^u_x$ and $\mu^s_x$ to be easily extended to $\mu^{cu}_x$ and $\mu^{cs}_x$. 
		The Bowen-Margulis measure $\mu$ of $\Phi$ on $M$ is then  constructed locally using the local product structure of the manifold. That is, at $x \in M$, choose open neighbourhoods of $U(x) \subset \mathcal{W}_\Phi^{cu}(x)$ and $V(x) \subset \mathcal{W}_\Phi^s(x)$, so that there is a well-defined map $\varphi \colon O(x) \to M$ which gives H\"{o}lder coordinates on the local product cube $O(x):=U(x) \times V(x)\subset M$. Fix an arbitrary $y \in U(x)$. For any open set $\Omega \subset O(x)$,  we let
		\begin{equation*} 
		\mu(\Omega):=\int_{z \in V(x)}\nu^{cu}\big(U(x)\times \{z\}\cap \Omega\big) d \nu_y(z),
		\end{equation*}
		where for any $\Omega'\subset V(x)$, we set $\nu_y(\Omega'):=\nu^{s} (\{y\}\times \Omega')$. By the invariance of $\nu^{s}$ under weak unstable holonomies, the previous definition is independent of the choice of $y \in U(x)$ and defines locally the Bowen-Margulis measure $\mu=\nu^{cu} \times \nu^s$. 
		\begin{prop}\label{prop marguli}
			If the measure of maximal entropy $\mu$ of $\Phi$ is absolutely continuous with respect to Lebesgue measure, then for any $x \in M$,   $\nu^{cu}$ is absolutely continuous with respect to Lebesgue measure on the weak unstable leaf $\mathcal{W}_\Phi^{cu}(x)$. Furthermore,  the density $e^{\psi}$ is H\"{o}lder continuous and smooth within the leaf $\mathcal{W}_\Phi^{cu}(x)$, and
			\begin{equation}\label{cohom smooth}
			\psi(\Phi^t(y)) - \psi(y) + ht = \log J^u_y(t), \quad \forall\, y \in \mathcal{W}_\Phi^{cu}(x),
			\end{equation}
			where $J^u_y(t)$ is the  Jacobian of the map $D \Phi^t|_{E_\Phi^u(y)} \colon E_\Phi^u(y) \to E_\Phi^u(\Phi^t(y))$. 
		\end{prop}
		
		\begin{proof}
			Fix $x \in M$ and choose a neighbourhood $O(x)=U(x)\times V(x)$ with local product structure  and let $\varphi$ be coordinates on $O(x)$ as described above.  Fix a point $p=\varphi(y,z)\in O(x)$. On the one hand, by the construction recalled previously, %we have 
			\begin{equation}\label{first marg}
			d\mu (p)=d \nu^{cu} (y)\otimes d\nu^s(z).
			\end{equation}  
			On the other hand, the measure  $\mu$ has local product structure, hence there exists a positive Borel function $\rho\colon U(x)\times V(x)$ such that 
			\begin{equation}\label{second leb}
			d\mu(p)= \rho(y,z) d \mu^{cu}_x(y)\otimes d\mu^s_x(z),
			\end{equation}
			where $\{\mu^{cu}_q\}_{q \in O(x)}$, resp. $\{\mu^{s}_q\}_{q \in O(x)}$ is a  system of conditional measures of $\mu$ for the foliation $\mathcal{W}_\Phi^{cu}$, resp. $\mathcal{W}_\Phi^{s}$. 
			As the foliations $\mathcal{W}_\Phi^{cu}$ and $\mathcal{W}_\Phi^{s}$ are absolutely continuous, for almost every $q \in O(x)$, the conditional measure $\mu^{cu}_q$, resp.  $\mu^{s}_q$ is absolutely continuous with respect to the Lebesgue  measure on the leaf $\mathcal{W}_\Phi^{cu}(q)$, resp. $\mathcal{W}_\Phi^{s}(q)$.  If we fix $z \in V(x)$, we deduce from \eqref{first marg}-\eqref{second leb} and the previous discussion that
			\begin{equation}\label{comp nu mu}
			d\nu^{cu} (y)= e^{\psi(y)} d y,
			\end{equation}
			where $dy$ denotes the Lebesgue measure on $U(x)\subset\mathcal{W}_\Phi^{s}(x)$. 
			Applying $(\Phi^t)_*$, it follows from \eqref{eq:margulis-cocycle} and \eqref{comp nu mu} that 
			\begin{equation}\label{comp nu mu bis}
			e^{ht} d \nu^{cu}(\Phi^t(y))= e^{ht+\psi(\Phi^t(y))} dy=e^{\psi(y)}  (\Phi^t)_*dy=e^{\psi(y)} J^u_y(t) dy,
			\end{equation}
			for some measurable function  $\psi\colon M \to \R$, where  $J^u_y(t)$ is the unstable Jacobian of $D \Phi^t$ at $y$. Thus, for almost every $y$,  it holds 
			$$
			\psi(\Phi^t(y)) - \psi(y) + ht =\log J^u_y(t).
			$$
			%Consider the conditional measure $\mu|_W$. Then the unstable foliation is a measurable foliation, and $\mu$ has a disintegration along it. Since $\mu$ is absolutely continuous with respect to Lebesgue, its measures $\mu_x^u$ are absolutely continuous with respect to Lebesgue at $\mu$-almost every $x \in W$. These must agree at $\mu$-almost every $x \in W$ by uniqueness of disintegrations. Since $x$ was arbtirary, we get that for $\mu$-almost every $x \in M$, $\mu_x^u$ is absolutlely continuous with respect to Lebesgue measure.
			%Let $\psi : M \to \R$ be defined so that $\mu^u_x = e^{\psi|_{W^u(x)}} \cdot \operatorname{Leb}$. Note that $\psi$ is a measurable function, and by \eqref{eq:margulis-cocycle}, $\psi(\varphi_t(y)) - \psi(y) + ht = J^u_y(t)$ at every density point $y$ of $\mu^u_x$. That is, 
			In other words, $\psi$ is a measurable transfer function making $\log J^u$ cohomologous to a constant. By Livsic's theorem, $\psi$ coincides almost everywhere with a H\"{o}lder solution which is smooth along the unstable leaves and H\"{o}lder transversally. With the upgraded regularity, we define an \textit{a priori} new family of conditionals along the unstable leaves at every point. By uniqueness of the family of measures satisfying \eqref{eq:margulis-cocycle} (up to multiplicative constant), these must coincide with $\nu^{cu}$ up to fixed scalar. Therefore, $\nu^{cu}$ is absolutely continuous with smooth density.
		\end{proof}
		
		\begin{remark}
		By Proposition \ref{prop marguli}, for any periodic point $y$ of period $ \mathcal{L}(y)>0$, taking $t=\mathcal{L}(y)$ in \eqref{cohom smooth}, we obtain  another proof of Proposition \ref{prop equal le} when $\Phi$ is topologically mixing.
		Conversely, if \eqref{logunsta} holds for any periodic orbit, then by Livsic's theorem, the $C^1$ cocycle 
		$$
		C\colon(y,t)\mapsto \log J^u_y(t)-ht
		$$
		 over $\Phi$ is a coboundary, and  \eqref{cohom smooth} follows. 
		\end{remark}

\end{document}